\documentclass[11pt, a4paper, dvipsnames]{article}

\usepackage{microtype}
\usepackage{graphicx}
\usepackage{subfigure}
\usepackage{booktabs} 

\usepackage{hyperref}
\usepackage[a4paper, total={6in, 8in}]{geometry}

\usepackage[hang,flushmargin]{footmisc}

\usepackage{amsmath}
\usepackage{amssymb}
\usepackage{mathtools}
\usepackage{amsthm}

\usepackage[capitalize,noabbrev]{cleveref}

\theoremstyle{plain}
\newtheorem{theorem}{Theorem}[section]
\newtheorem{propo}[theorem]{Proposition}
\newtheorem{lemma}[theorem]{Lemma}
\newtheorem{corollary}[theorem]{Corollary}
\theoremstyle{definition}

\newtheorem{assumption}{Assumption}

\crefname{assumption}{Assumption}{assumptions}

\newtheorem{assumptionalt}{Assumption}[assumption]

\newenvironment{assumptionp}[1]{
  \renewcommand\theassumptionalt{#1}
  \assumptionalt
}{\endassumptionalt}

\theoremstyle{remark}
\newtheorem{remark}[theorem]{Remark}
\newtheorem{example}[theorem]{Example}

\usepackage{tikz}
\usetikzlibrary{positioning}
\usetikzlibrary{shapes}
\usetikzlibrary{scopes}

\usepackage{graphicx}
\graphicspath{{./img}}

\usepackage{enumitem}

\usepackage{mathrsfs}

\usepackage{blindtext}

\usepackage{dsfont} 

\global\long\def\esp{\mathbb{E}}%
\global\long\def\F{\mathcal{F}}%
\global\long\def\R{\mathbb{R}}%
\global\long\def\P{\mathbb{P}}%

\newcommand{\imp}{\mathrm{imp}}
\newcommand{\mis}{\mathrm{mis}}

\newcommand{\argmin}{\mathop{\mathrm{arg \, min}}}

\usepackage[textsize=tiny]{todonotes}

\newcommand{\cb}[1]{}
\newcommand{\al}[1]{}
\newcommand{\ad}[1]{}
\newcommand{\es}[1]{}

\usepackage[normalem]{ulem}
\usepackage[round]{natbib}
\usepackage{hyperref}
\hypersetup{
    colorlinks = true,
    linkcolor = {PineGreen},
    citecolor = {RoyalBlue},
    urlcolor = Blue
}

\definecolor{OliveGreen}{rgb}{0.24, 0.71, 0.54}
\definecolor{RoyalBlue}{rgb}{0.0, 0.47, 0.75}
\definecolor{BrickRed}{rgb}{0.77, 0.12, 0.23}
\definecolor{Vert}{RGB}{0,128,0}

\usepackage{natbib}

\begin{document}

\title{
Naive imputation implicitly regularizes high-dimensional linear models.
}

\author{Alexis Ayme, Claire Boyer, Aymeric Dieuleveut \\ \& Erwan Scornet}



\date{}





\maketitle

\begin{abstract}
Two different approaches exist to handle missing values for prediction: either imputation, prior to fitting any predictive algorithms, or dedicated methods able to natively incorporate missing values.
While imputation is widely (and easily) use, it is unfortunately biased when low-capacity predictors (such as linear models) are applied afterward. However, in practice, naive imputation exhibits good predictive performance. In this paper, we study the impact of imputation in a high-dimensional linear model with MCAR missing data.  We prove that zero imputation performs an implicit regularization closely related to the ridge method, often used in high-dimensional problems. Leveraging on this connection, we establish that the imputation bias is controlled by a ridge bias, which vanishes in high dimension. As a predictor, we argue in favor of the averaged SGD strategy, applied to zero-imputed data. We establish an upper bound on its generalization error, highlighting that imputation is benign in the $d \gg \sqrt{n}$ regime. Experiments illustrate our findings. 
\end{abstract}


\section{Introduction}

Missing data has become an inherent problem in modern data science. Indeed, most real-world data sets contain missing entries due to a variety of reasons: merging different data sources, 
sensor failures, difficulty to collect/access data in sensitive fields (e.g., health), just to name a few. 
The simple, yet quite extreme, solution of throwing partial observations away can drastically reduce the data set size and thereby hinder further statistical analysis. Specific  methods should be therefore developed to handle missing values. 
Most of them are dedicated to model estimation, aiming at inferring  the underlying model parameters despite missing values  \citep[see, e.g.,][]{RUBIN76}. 
In this paper, we take a different route and consider a supervised machine learning (ML) problem with missing values in the training and test inputs, for which our aim is to build a  prediction function (\textit{and not} to estimate accurately the true model parameters). 

\paragraph{Prediction with NA}
A common practice to perform supervised learning with missing data is to simply impute the data set first, and then train any predictor on the completed/imputed data set. 
The imputation technique can be simple (e.g., using mean imputation) or more elaborate \citep{van2011mice,yoon2018gain,muzellec2020missing,ipsen2022deal}.
While such widely-used two-step strategies lack deep theoretical foundations, they have been shown to be consistent, provided that the approximation capacity of the chosen predictor is large enough  \citep[see][]{josse2019consistency,le2021sa}. 
When considering low-capacity predictors, such as linear models, other theoretically sound strategies consist of decomposing the prediction task with respect to all possible missing patterns  \citep[see][]{le2020linear,ayme2022near} or by automatically detecting relevant patterns to predict, thus breaking the combinatorics of such pattern-by-pattern predictors \citep[see the specific NeuMiss architecture in][]{lemorvan:hal-02888867}. 
Proved to be nearly optimal \citep{ayme2022near}, such approaches are likely to be robust to very pessimistic missing data scenarios. Inherently, they do not scale with high-dimensional data sets, as the variety of missing patterns explodes. 
Another direction is  advocated in \citep{agarwal2019robustness} relying on principal component regression (PCR) in order to train linear models with missing inputs. However, out-of-sample prediction in such a case requires to retrain the predictor on the training and test sets (to perform a global PC analysis), which  strongly departs from classical ML algorithms massively used in practice.

In this paper, we focus on the high-dimensional regime of linear predictors, which will appear to be more favorable to handling missing values via simple and cheap imputation methods, in particular in the missing completely at random (MCAR) case.

\paragraph{High-dimensional linear models} 

In supervised learning with complete inputs, when training a parametric method (such as a linear model) in a high-dimensional framework, one often resorts to an $\ell^2$ or ridge regularization technique. 
On the one hand, such regularization  fastens the optimization procedure (via its convergence rate) \citep{dieuleveut2017harder}; on the other hand, it also improves the generalization capabilities of the trained predictor \citep{caponnetto2007optimal,hsu2012random}. 
In general, this second point holds  for explicit $\ell^2$-regularization, but some works also emphasize the ability of optimization algorithms to induce an implicit regularization, e.g., via early stopping \citep{yao2007early} and more recently via gradient strategies in interpolation regimes \citep{bartlett2020benign,chizat2020implicit,pesme2021implicit}. 

\paragraph{Contributions}
For supervised learning purposes, we consider a zero-impu\-ta\-tion strategy consisting in  replacing input missing entries by zero, and we formalize the induced bias on a regression task (Section~\ref{sec:setting}). 
When the missing values are said Missing Completely At Random (MCAR), we prove that zero imputation, used prior to training a linear model, introduces an implicit regularization closely related to that of ridge regression (Section~\ref{sec:bias}). 
This bias is exemplified to be negligible in settings commonly encountered in high-dimensional regimes, e.g., when the inputs admit a low-rank covariance matrix. 
We then advocate for the choice of 
an averaged stochastic gradient algorithm (SGD) applied on zero-imputed data (Section~\ref{sec:sgd}). Indeed, such a predictor, being computationally efficient, remains particularly relevant for high-dimensional learning.
For such a strategy, we establish a generalization bound valid for all $d,n$, in which the impact of imputation on MCAR data is soothed when $d \gg \sqrt{n}$. 
These theoretical results legitimate the widespread imputation approach, adopted by most practitioners, and are corroborated by numerical experiments in Section~\ref{sec:experiments}. All proofs are to be found in the Appendix. 


\section{Background and motivation}
\label{sec:setting}

\subsection{General setting and notations}
\label{sec:subsec_setting}

In the context of supervised learning, consider $n\in \mathbb N$ input/output observations $((X_i,Y_i))_{i\in [n]}$, i.i.d.~copies of a generic pair $(X,Y)\in \mathbb{R}^d \times \mathbb{R}$. By some abuse of notation, we always use $X_i$ with $i\in [n]$ to denote the $i$-th observation living in $\mathbb{R}^d$, and $X_j$ (or $X_k$)  with $j\in [d]$ (or $k\in [d]$) to denote the $j$-th (or $k$-th) coordinate of the generic input $X$ (see Section~\ref{app:notations} for notations).  

\paragraph{Missing values}
In real data sets, the input covariates $(X_i)_{i\in[n]}$ are often only partially observed. To code for this missing information,  we introduce the random vector $P\in \{0,1\}^d$, referred to as mask or missing pattern, and such that $P_{j}=0$ if the $j$-th coordinate of $X$, $X_{j}$, is missing and $P_{j}=1$ otherwise. The random vectors $P_1, \hdots, P_n$ are assumed to be i.i.d.~copies of a generic random variable $P\in \{0,1\}^d$ and the missing patterns of $X_1,\hdots , X_n$. 
%
Note that we assume that the output is always observed and only entries of the input vectors can be missing. Missing data are usually classified into 3 types, initially introduced by \citep{RUBIN76}. In this paper, we focus on the MCAR assumption where missing patterns and (underlying) inputs are independent.  
\begin{assumption}[Missing Completely At Random - MCAR]\label{ass:MCAR} 
{The pair $(X,Y)$ and the missing pattern $P$ associated to $X$ are independent.}
\end{assumption}
For $j\in [d]$, we define $\rho_j:=\P(P_j=1)$, i.e., $1-\rho_j$ is the  expected proportion of missing values on the $j$-th feature. 
A particular case of MCAR data requires, not only the independence of the mask and the data, but also the independence between all mask components, as follows.
\begin{assumptionp}{\ref{ass:MCAR}'}[Ho-MCAR: MCAR pattern with independent homogeneous components]\label{ass:bernoulli_model}
{The pair $(X,Y)$ and the missing pattern $P$ associated to $X$ are independent}, and    
the distribution of $P$ satisfies $P\sim\mathcal{B}(\rho)^{\otimes d}$ for $0< \rho \leq 1$, with $1-\rho$ the expected proportion of missing values, and $\mathcal{B}$ the Bernoulli distribution.
\end{assumptionp}

\paragraph{Naive imputation of covariates} A common way to handle missing values for any learning task is to first impute missing data, to obtain a complete dataset, to which standard ML algorithms can then be applied. 
In particular, constant imputation (using the empirical mean or an oracle constant provided by experts) is very common among practitioners. In this paper, we consider, even for noncentered distributions, the naive imputation by zero,
so that the imputed-by-0 observation $(X_{\rm{imp}})_i$, for $i\in [n]$, is given by
\begin{equation}\label{eq:Xtilde}
    (X_{\rm{imp}})_i=P_i\odot X_i.
\end{equation}

\paragraph{Risk} 
Let $f:\R^d\to\R$ be a measurable prediction function, 
based on a complete $d$-dimensional input. 
Its predictive performance can be measured through its quadratic risk, 
\begin{equation}
    \label{eq:risk}
    R(f):= \esp\left[\left(Y-f\left(X\right)\right)^2\right].
\end{equation}
Accordingly, we let $f^\star(X)=\esp[Y|X]$ be the Bayes predictor for the complete case and $R^\star$ the associated risk. 

In the presence of missing data, one can still use the predictor function $f$, applied to the imputed-by-0 input $X_{\imp}$, resulting in the prediction  $f(X_{\imp})$. In such a setting, the risk of $f$, acting on the imputed data, is defined by
\begin{equation}
    \label{eq:risk0}
    R_{\imp}(f):= \esp\left[\left(Y-f(X_{\imp})\right)^2\right].
\end{equation}

For the class $\F$ of linear prediction functions from $\mathbb{R}^d$ to $\mathbb{R}$, we respectively define
\begin{equation}\label{eq:risk_on_a_class}
    R^\star(\F)= \inf_{f\in\F}R(f),
\end{equation}
and
\begin{equation}\label{eq:risk_on_a_class0}
    R_{\imp}^\star(\F)= \inf_{f\in\F}R_{\imp}(f),
\end{equation}
as the infimum over the class $\F$ with respectively complete and imputed-by-0 input data. 

For any linear prediction function defined by $f_\theta(x)=\theta^\top x$ for any $x\in\R^d$ and a fixed $\theta\in\R^d$, as $f_\theta$ is completely determined by the parameter $\theta$, we make the abuse of notation of $R(\theta)$ to designate $R(f_\theta)$ (and $R_{\imp}(\theta)$ for $R_{\imp}(f_\theta)$). We also let $\theta^{\star}\in \R^d$ (resp. $\theta_{\imp}^\star$) be a parameter achieving the best risk on the class of linear functions, i.e., such that $R^\star(\F)=R(\theta^{\star})$ (resp. $R^\star_{\imp}(\F)=R_{\imp}(\theta_{\imp}^{\star})$).

\paragraph{Imputation bias} Even if the prepocessing step consisting of imputing the missing data by $0$ is often used in practice, this imputation technique can introduce a bias in the prediction. 
We formalize this \textit{imputation bias} as 
\begin{equation}\label{eq:bias0}
    B_{\imp}(\F):= R_{\imp}^\star(\F)-R^\star(\F).
\end{equation}
This quantity represents the difference in predictive performance between the best predictor  on complete data and that on 
imputed-by-$0$ inputs. 
In particular, if this quantity is small, the risk of the best predictor on imputed data is close to that of the best predictor when all data are available.  
Note that, in presence of missing values, one might be interested in the Bayes predictor 
\begin{equation}
    f_{\mis}^\star(X_{\imp},P)=\esp[Y|X_{\imp},P].
\end{equation}
and its associated risk 
 $R^\star_{\mathrm{mis}}$.
\begin{lemma}\label{lem:RmisVsR}
Assume that regression model $Y=f^\star(X)+\epsilon$ is such that $\epsilon$ and $P$ are independent, 
then  $R^{\star}\leq R_{\mathrm{mis}}^{\star}$.
\end{lemma}
Intuitively, under the classical assumption $\varepsilon \perp\!\!\!\perp P$ \citep[see][]{josse2019consistency}, which is a verified under Assumption~\ref{ass:MCAR},  missing data ineluctably deteriorates the original prediction problem.
As a direct consequence, for a well-specified linear model on the complete case $f^\star\in\F$, 
\begin{equation}\label{eq:excessRiskBound}
    R_{\imp}(\mathcal{F})- R^\star_{\mis}\leq 
    B_{\imp}(\F) .
\end{equation}
Consequently, in this paper, we focus our analysis on the bias (and excess risk) associated to impute-then-regress strategies with respect to the complete-case problem (right-hand side term of \eqref{eq:excessRiskBound}) thus controlling the excess risk of imputation with respect to the missing data scenario (left-hand side term of \eqref{eq:excessRiskBound}).

%
In a nutshell, the quantity $B_{\imp}(\F)$ thus represents how missing values, handled with zero imputation, increase the difficulty of the learning problem. 
This effect can be tempered in a high-dimensional regime, as rigorously studied in Section~\ref{sec:bias}. 
To give some intuition, let us now study the following toy
example. 
\begin{example}\label{ex:Xtous_egals}
Assume an extremely redundant setting in which all covariates are equal, that is, for all $j\in[d]$, $X_j=X_1$  with $\esp\left[X_1^2\right]=1$. Also assume that the output is such that $Y=X_1$ and that  \cref{ass:bernoulli_model} holds with $\rho=1/2$. In this scenario, due to the input redundancy, all $\theta$ satisfying $\sum_{j=1}^d \theta_j=1$ minimize $\theta \mapsto R(\theta)$. Letting, for example,  $\theta_1=(1,0,...,0)^\top$, we have $R^\star = R(\theta_1) =0$ but
\begin{equation*}
    R_{\imp}(\theta_1)= \esp\left[(X_1-P_1X_1)^2\right]= \frac{1}{2}.
\end{equation*}
This choice of $\theta_1$ introduces an irreducible discrepancy between the risk computed on the imputed data and the Bayes risk $R^\star = 0$. Another choice of parameter could actually help to close this gap. 
Indeed, by exploiting the redundancy in covariates, the parameter $\theta_2=(2/d,2/d,...,2/d)^\top$ (which is not a minimizer of the initial risk anymore) gives
\begin{align*}
    R_{\imp}(\theta_2)&= \esp\bigg[\Big(X_1-\frac{2}{d}\sum_{j=1}^d P_j X_j\Big)^2\bigg]
    = \frac{1}{d},
\end{align*}
so that the imputation bias $B_{\imp}(\F)$ is bounded by $1/d$, tending to zero as the dimension increases. Two other important observations on this example follow. First, this bound is still valid if $\esp X_1 \neq 0$, thus the imputation by $0$ is still relevant even for non-centered data. Second, we remark that $\Vert \theta_2\Vert_2^2=4/d$, thus good candidates to predict with imputation seem to be of small norm in high dimension. This will be proved for more general settings, in Section~\ref{sec:sgd}. 
\end{example}
The purpose of this paper is to generalize the phenomenon described in \Cref{ex:Xtous_egals} to less stringent settings. In light of this example, we focus our analysis on scenarios for which some information is shared across input variables: for linear models, correlation plays such a role.

\paragraph{Covariance matrix} For a generic complete input $X\in \R^d$, call $\Sigma:=\esp\left[XX^\top\right]$ the associated covariance matrix, 
admitting the following singular value decomposition
\begin{equation}\label{eq:SigmaSVD}
    \Sigma= \sum_{j=1}^d \lambda_jv_jv_j^\top,
\end{equation}
where $\lambda_j$ (resp.\ $v_j$) are singular values (resp.\ singular vectors) of $\Sigma$ and such that $\lambda_1\geq...\geq\lambda_d$. The associated pseudo-norm is given by, for all $\theta\in\R^d$,
\begin{equation*}
    \Vert \theta\Vert_\Sigma^2:=\theta^\top\Sigma\theta=\sum_{j=1}^d \lambda_j(v_j^\top\theta)^2.
\end{equation*}
For the best linear prediction, $Y= X^\top\theta^{\star}+\epsilon$, and the noise satisfies $\esp[\epsilon X]=0$ (first order condition). Denoting $\esp[\epsilon^2]=\sigma^2$, we have 
\begin{equation}\label{eq:espY^2}
    \esp Y^2=\Vert \theta^\star\Vert_\Sigma^2+\sigma^2= \sum_{j=1}^d \lambda_j(v_j^\top\theta^{\star})^2+\sigma^2.
\end{equation}
The quantity $\lambda_j(v_j^\top\theta^{\star})^2$ can be therefore interpreted as the part of the variance explained by the singular direction $v_j$.  
\begin{remark}
Note that, in the setting of \cref{ex:Xtous_egals}, $\Sigma$ has a unique positive singular values $\lambda_1=d$, that is to say, all of the variance is concentrated on the first singular direction. Actually, our analysis will stress out that a proper decay of singular values leads to low imputation biases. 
\end{remark}
Furthermore, for the rest of our analysis, we need the following assumptions on the second-order moments of $X$.
\begin{assumption}\label{ass:2moment_sup}
    $\exists L < \infty$ such that, $\forall j\in[d]$, $\esp[X_j^2]\leq L^2$. 
\end{assumption}
\begin{assumption}\label{ass:2moment_inf}
    $\exists \ell >0$ such that, $\forall j\in[d]$, 
    $\esp[X_j^2]\geq \ell^2$. 
\end{assumption}
For example, Assumption \ref{ass:2moment_sup} and \ref{ass:2moment_inf} hold with $L^2=\ell^2=1$ with normalized data.

 

\section{Imputation bias for linear models}\label{sec:bias}

\subsection{Implicit regularization of imputation}
\label{subsec:implicit_regularization}

Ridge regression, widely used in high-dimensional settings, and notably for its computational purposes,  amounts to form an $\ell_2$-penalized version of the least square estimator: 
\[
\hat{\theta}_{\lambda}\in\argmin_{\theta\in\R^d}\left\{ \frac{1}{n}\sum_{i=1}^{n}\left(Y_{i}-f_\theta(X_{i})\right)^{2}+\lambda\left\Vert \theta\right\Vert _{2}^{2}\right\}, 
\]
where $\lambda> 0$ is the penalization parameter. The associated generalization risk can be written as
\[
R_{\lambda}(\theta):= R(\theta)+\lambda\left\Vert \theta\right\Vert _{2}^{2}. 
\]
\Cref{prop:regularization_implicite} establishes a link between imputation and ridge penalization.
\begin{propo}\label{prop:regularization_implicite} Under \Cref{ass:MCAR}, let $V$ be the covariance matrix of $P$ ($V_{ij}=\mathrm{Cov}(P_i,P_j)$) and $H=\mathrm{diag}(\rho_1,\dots,\rho_d)$, with $\rho_j=\P(P_j=1)$. Then, for all $\theta$, 
    \[
R_{\imp}(\theta)=R\left(H\theta\right)+\left\Vert \theta\right\Vert _{V\odot\Sigma}^{2}.
\]
In particular, under Assumptions~\ref{ass:bernoulli_model}, \ref{ass:2moment_sup} and \ref{ass:2moment_inf} when $L^2=\ell^2$,
\begin{align}
R_{\imp}(\theta)=R\left(\rho\theta\right)+L^2\rho(1-\rho)\left\Vert \theta\right\Vert _{2}^{2}.
\label{eq:pen_ridge_ref_equality_case}
\end{align}
\end{propo}
This result highlights the implicit $\ell^2$-regularization at work: performing standard regression on zero-imputed ho-MCAR data can be seen as performing a ridge regression on complete data, whose strength $\lambda$ depends on the missing values proportion. 
More precisely, using Equation \eqref{eq:pen_ridge_ref_equality_case}, the optimal predictor $\theta_{\imp}^\star$
working with imputed samples verifies 
\begin{equation*}
    \theta_{\imp}^\star = \frac{1}{L^2\rho}\argmin_{\theta\in \R^d}\left\{R\left(\theta\right)+\lambda_{\imp}\left\Vert \theta\right\Vert _{2}^{2}\right\},
\end{equation*}
with $\lambda_{\imp}:=L^2\left( \frac{1-\rho}{\rho} \right)$. 
We exploit this correspondence in \Cref{sec:bias_and_Sigma} and \ref{sec:biasBernoullihetero} to 
control the imputation bias.



\subsection{Imputation bias for linear models with ho-MCAR missing inputs}\label{sec:bias_and_Sigma}

When the inputs admit ho-MCAR missing patterns (\cref{ass:bernoulli_model}), the zero-imputation bias $B_{\imp}(\F)$ 
induced in the linear model is controlled by a particular instance of the ridge regression bias \citep[see, e.g.,][]{hsu2012random,dieuleveut2017harder,mourtada2019exact}, defined in general by
\begin{align}\label{eq:def_Bridge}
B_{{\rm {ridge},\lambda}}(\F)  &:=\inf_{\theta\in\R^{d}}\left\{ R_\lambda (\theta)-R^\star(\F)\right\}\\&=\lambda \left\Vert\theta^{\star}\right\Vert _{\Sigma(\Sigma+\lambda I)^{-1}}^{2}. 
\end{align}
\begin{theorem}\label{prop:biasBernoulliHomo}
Under Assumption~\ref{ass:bernoulli_model}, \ref{ass:2moment_sup}, and \ref{ass:2moment_inf}, one has
\begin{equation*}
 B_{{\rm {ridge},\lambda_{\imp}'}}(\F) \leq B_{\imp}(\F)\leq B_{{\rm {ridge},\lambda_{\imp}}}(\F),    
\end{equation*}
with $\lambda_{\imp}':= \ell^2 \left(\frac{1-\rho}{\rho}\right) $ and $\lambda_{\imp} = L^2 \left(\frac{1-\rho}{\rho}\right)$. 

\end{theorem}

As could be expected from \Cref{prop:regularization_implicite}, the zero-imputation bias is {lower and upper-bounded} by the ridge bias, with a  penalization constant depending on the fraction of missing values.  
In the specific case where $ \ell^2=L^2$ (same second-order moment), the imputation bias exactly equals a ridge bias with a constant $L^2(1 - \rho)/\rho$. Besides, in the extreme case where there is no missing data ($\rho=1$) then $\lambda_{\imp}=0$,  and the bias vanishes. On the contrary, if there is a large percentage of missing values ($\rho \to 0$) then  $\lambda'_{\imp} \to +\infty$ and the imputation bias amounts to the excess risk of the naive predictor, i.e., $B_{\imp}(\F)=R(0_{\R^d})-R^\star(\F)$. For the intermediate case where half of the data is likely to be missing ($\rho=1/2$), we obtain $\lambda_{\imp}=L^2$.

Thus, in terms of statistical guarantees, performing linear regression on imputed inputs suffers from a bias comparable to that of a ridge penalization, but with a fixed hyperparameter $\lambda_{\imp}$. Note that, when performing standard ridge regression in a high-dimensional setting, the best theoretical choice of the penalization parameter usually scales as $d/n$ \citep[see][for details]{sridharan2008fast,hsu2012random,mourtada2022elementary}. 
If $\rho \gtrsim L^2\frac{n}{d+n}$ (which is equivalent to $\lambda_{\imp}\lesssim \frac{d}{n}$), the imputation bias remains smaller than that of the ridge regression with the optimal hyperparameter $\lambda = d/n$ (which is commonly accepted in applications). 
In this context, performing zero-imputation prior to applying a ridge regression allows handling easily missing data without drastically increasing the overall bias.


In turns out that the bias of the ridge regression in random designs, and thus the imputation bias, can be controlled, under 
classical assumptions about low-rank covariance structures
\citep[][]{caponnetto2007optimal,hsu2012random,dieuleveut2017harder}. {In all following examples, we consider that $\mathrm{Tr}(\Sigma)=d$, which holds in particular for normalized data. }

\begin{example}[Low-rank covariance matrix with equal singular values]\label{ex:low_rank}
Consider a covariance matrix with a low rank $r\ll d$ and constant eigenvalues ($\lambda_1=\dots=\lambda_r=\frac{d}{r}$). Then $\Sigma(\Sigma+\lambda_{\imp} I)^{-1}\preceq \lambda_r^{-1} \Sigma=\frac{r}{d}\Sigma$ and \cref{prop:biasBernoulliHomo} leads to 
\begin{equation*}
    B_{\imp}(\F)\leq\lambda_{\imp}\frac{r}{d}\left\Vert \theta^{\star}\right\Vert _{\Sigma}^{2}.
\end{equation*}
Hence, the imputation bias is small when  $r\ll d$ (low-rank setting). 
Indeed, for a fixed dimension, when the covariance is low-rank, there is a lot of redundancy across variables, which helps counterbalancing missing information in the input variables, thereby reducing the prediction bias.
\end{example}
Note that
\cref{ex:low_rank} ($r\ll d$) is a generalization of \cref{ex:Xtous_egals} (in which $r=1$), and is rotation-invariant contrary to the latter. 

\begin{remark}
A first order condition (see equation \eqref{eq:espY^2}) implies that $ \Vert \theta^\star\Vert_\Sigma^2+\sigma^2 = \esp Y^2=R\left(0_{\R^d}\right)$, which is independent of the dimension~$d$. Thus, in all our upper bounds, $\Vert \theta^\star\Vert_\Sigma^2$ can be replaced by $\esp Y^2$, which is dimension-free. Consequently, we can interpret \Cref{ex:low_rank} (and the following examples) upper bound as follows: if $r\ll d$, then  the risk of the naive predictor is divided by $d/r \gg 1$. As a consequence, $B_{\imp}$ tends to zero when the dimension increases and the rank is fixed. 
\end{remark}

\begin{example}[Low-rank covariance matrix compatible with $\theta^\star$ ]\label{ex:low_rank2}
Consider a covariance matrix with a low rank $r\ll d$ and assume that $\langle \theta^\star,v_1\rangle^2\geq \dots\geq \langle \theta^\star,v_d\rangle^2 $ (meaning that $\theta^\star$ is well represented with the 
 first eigendirections of $\Sigma$), \cref{prop:biasBernoulliHomo} leads to 
\begin{equation*}
    B_{\imp}(\F)\lesssim \lambda_{\imp}\frac{r(\log(r)+1)}{d}\left\Vert \theta^{\star}\right\Vert _{\Sigma}^{2}.
\end{equation*}
This result is similar to \cref{ex:low_rank} (up to a log factor), except that assumptions on the eigenvalues of $\Sigma$ have been replaced by a condition on the compatibility between the covariance structure and $\theta^\star$.
If $\theta^\star$ is well explained by the largest eigenvalues then the imputation bias remains low. 
This underlines that imputation bias does not only depend on the spectral structure of $\Sigma$ but also on $\theta^\star$. 
\end{example}

\begin{example}[Spiked model, \citet{johnstone2001}]\label{ex:noised_low_rank}
In this model, the  covariance matrix can be decomposed as $\Sigma= \Sigma_{\leq r}+ \Sigma_{> r} $ where $\Sigma_{\leq r}$ corresponds to the low-rank part of the data with large eigenvalues and $\Sigma_{> r}$  to the residual high-dimensional data. Suppose that $\Sigma_{> r}\preceq \eta I$ (small operator norm) and that all non-zero eigenvalues of $\Sigma_{\leq r}$ are equal, then   \cref{prop:biasBernoulliHomo} gives
\begin{equation*}
    B_{\imp}(\F)\leq\frac{\lambda_{\imp}}{1-\eta}\frac{r}{d}\left\Vert \theta^{\star}\right\Vert _{\Sigma}^{2}+\eta \left\Vert \theta^{\star}_{>r}\right\Vert _{2}^{2}, 
\end{equation*}
where $\theta^{\star}_{>r}$ is the projection of $\theta^{\star}$ on the range of $\Sigma_{> r}$. Contrary to \cref{ex:low_rank}, $\Sigma$ is only \textit{approximately} low rank, 
and one can refer to $r$ as the ``effective rank" of $\Sigma$ \citep[see][]{bartlett2020benign}. 
The above upper bound admits a term in $O(r/d)$ (as in \cref{ex:low_rank}), but also suffers from a non-compressible part $\eta \left\Vert \theta^{\star}_{>r}\right\Vert _{2}^{2}$, due to the presence of residual (potentially noisy) high-dimensional data. Note that, if $\theta^{\star}_{>r}=0$ (only the low-dimensional part of the data is informative) then we retrieve the same rate as in \cref{ex:low_rank}.
\end{example}


\subsection{Imputation bias for linear models and general MCAR settings}
\label{sec:biasBernoullihetero}

\cref{prop:biasBernoulliHomo} holds only for Ho-MCAR settings, which excludes the case of dependence between mask components. To cover the case of dependent variables $P_1, \hdots , P_d$ under \cref{ass:MCAR}, recall $\rho_j:=\P(P_j=1)$ the probability that the component $j$ is not missing, and define the matrix $C \in \mathbb{R}^{d\times d}$ associated to $P$, given by:
\begin{equation}\label{eq:corrMask} 
    C_{kj}:= \frac{V_{k,j}}{\rho_k\rho_j}, \quad (k,j)\in [d]\times [d].
\end{equation}
Furthermore, under \Cref{ass:2moment_sup}, define
\begin{equation}\label{eq:defKM}
\Lambda_{\imp}:= L^2\lambda_{\rm{max}}(C). 
\end{equation}
The following result  establishes an upper bound on the imputation bias for general MCAR settings. 
\begin{propo}\label{prop:biasMCAR}
Under \Cref{ass:MCAR} and \ref{ass:2moment_sup}, we have
\[
B_{\imp}(\F)\leq B_{{\rm {ridge},\Lambda_{\imp}}}(\F).
\]
\end{propo}
The bound on the bias is similar to the one of  \cref{prop:biasBernoulliHomo} but appeals to $\lambda = \Lambda_{\imp}$ which takes into account the correlations between the components of  missing patterns. Remark that, under \cref{ass:bernoulli_model}, there are no correlation and $\Lambda_{\imp}=L^2\frac{1-\rho}{\rho}$, thus matching the result in \cref{prop:biasBernoulliHomo}. The following examples highlight generic scenarios in which an explicit control on $\Lambda_{\imp}$ is obtained.
\begin{example}[Limited number of correlations]\label{ex:LimitedCorr}
If each missing pattern component is correlated with  at most $k-1$ other components then $\Lambda_{\imp}\leq L^2 k\max_{j\in[d]}\left\{\frac{1-\rho_j}{\rho_j}\right\} $.  
\end{example}

\begin{example}[Sampling without replacement]\label{ex:samplingWithoutRempl}
Missing pattern components are sampled as $k$ components without replacement in $[d]$, then $\Lambda_{\imp}=L^2 \frac{k+1}{d-k}$. In particular, if one half of data is missing ($k=\frac{d}{2}$) then $\Lambda_{\imp}
\leq 3L^2$. 
\end{example}

In conclusion, we proved that the  imputation bias is controlled by the ridge bias, with a penalization constant $\Lambda_{\imp}$, under any MCAR settings. More precisely, all examples of the previous section (\Cref{ex:low_rank,ex:low_rank2,ex:noised_low_rank}), relying on a specific structure of the covariance matrix $\Sigma$ and the best predictor $\theta^{\star}$, are still valid, replacing $\lambda_{\imp}$
 by $\Lambda_{\imp}$. Additionally, specifying the missing data generation (as in \Cref{ex:samplingWithoutRempl,ex:LimitedCorr}) allows us to control the imputation bias, which is then proved to be small in high dimension, for all the above examples. 

\section{SGD on zero-imputed data}\label{sec:sgd}

Since the imputation bias is only a part of the story, we need to propose a proper estimation strategy for
$\theta^{\star}_{\imp}$. 
To this aim, we choose to train a linear predictor on imputed samples, using an averaged stochastic gradient algorithm \citep{polyak1992acceleration}, described below. We then establish  generalization bounds on the excess risk of this estimation strategy. 

\subsection{Algorithm}
\label{subsec:algo}


Given an initialization $\theta_0\in\R^d$ and a constant learning rate $\gamma >0$, the iterates of the averaged SGD algorithm are given at iteration $t$ by 
\begin{equation}\label{eq:SGDiteration}
 \theta_{\imp,t}=\left[I-\gamma X_{\imp,t}X_{\imp,t}^{\top}\right]\theta_{\imp,t-1}+\gamma Y_{t}X_{\imp,t},   
\end{equation} 
so that after one pass over the data (early stopping), the final estimator $\bar\theta_{\imp,n}$  is given by the
Polyak-Ruppert average $
\bar\theta_{\imp,n}=\frac{1}{n+1}\sum_{t=1}^n \theta_{\imp,t}$.
Such recursive procedures are suitable for high-dimensional settings, and indicated for model miss-specification (induced here by missing entries), as studied in \citet{bach2013non}. Besides, they are very competitive for large-scale datasets, since one pass over the data 
requires $O(dn)$ operations. 

\subsection{Generalization bound}

Our aim is to derive a generalization bound on the predictive performance of the above algorithm, trained on zero-imputed data. To do this, we require the following extra assumptions on the complete data.
\begin{assumption}\label{ass:SGD}
There exist $\sigma>0$ and $\kappa>0$ such that $\esp[XX^{\top}\left\Vert X\right\Vert _{2}^{2}]\preceq \kappa \mathrm{Tr}(\Sigma)\Sigma$ and $\esp [\epsilon^{2}\left\Vert X\right\Vert _{2}^{2}]\leq\sigma^{2}\kappa\mathrm{Tr}(\Sigma)$, where $\epsilon= Y-X^\top\theta^\star$. 
\end{assumption}
 \Cref{ass:SGD} is a classical fourth-moment assumption in stochastic optimization  \citep[see][for details]{bach2013non,dieuleveut2017harder}. Indeed, the first statement in \Cref{ass:SGD} holds, for example, if $X$ is a Gaussian vector (with $\kappa=3$) or when $X$ satisfies $\left\Vert X\right\Vert _{2}\leq \kappa \mathrm{Tr}(\Sigma) $ almost surely. The second statement in Assumption~\ref{ass:SGD} holds, for example, if the model is well specified or when the noise $\varepsilon$ is almost surely bounded. Note that if the first part holds then the second part holds with $\sigma^2\leq 2\esp[Y^2]+2\esp[Y^4]^{1/2}$. 

Our main result, establishing an upper bound on the risk of SGD applied to zero-imputed data, follows. 
\begin{theorem}\label{thm:SGD_bound}
Under \Cref{ass:SGD}, choosing a constant learning rate $\gamma = \frac{1}{\kappa\mathrm{Tr}(\Sigma)\sqrt{n}}$ leads to 
\begin{align*}
    \esp\left[R_{\imp}\left(\bar{\theta}_{\imp,n}\right)\right]-R^{\star}(\F)  \lesssim & ~ \frac{\kappa\mathrm{Tr}(\Sigma)}{\sqrt{n}}\left\Vert \theta^{\star}_{\imp}-\theta_0\right\Vert _{2}^{2} +\frac{\sigma^2+\left\Vert \theta^{\star}\right\Vert _{\Sigma}^{2}}{\sqrt{n}} + B_{\imp}(\F),
\end{align*}
where $\theta^\star$ (resp. $\theta^{\star}_{\imp}$) is the best linear predictor for complete (resp. with imputed missing values) case. 
\end{theorem}
Theorem~\ref{thm:SGD_bound} gives an upper bound on the difference between the averaged risk  $\esp[R_{\imp}\left(\bar{\theta}_{\imp,n}\right)]$ of the estimated linear predictor with imputed missing values (in both train and test samples) and $R^{\star}(\F)$, the risk of the best linear predictor on the complete case.  
Interestingly, by Lemma~\ref{lem:RmisVsR} and under a well-specified linear model, the latter also holds for $\esp\left[R_{\imp}\left(\bar{\theta}_{\imp,n}\right)\right]-R_{\mathrm{mis}}^{\star}$.
The generalization bound in Theorem~\ref{thm:SGD_bound} takes into account the statistical error of the method as well as the optimization error. More precisely, the upper bound can be decomposed into 
$(i)$ a bias associated to the initial condition, 
$(ii)$ a variance term of the considered method, and
$(iii)$ the aforementioned imputation bias.


The variance term $(ii)$ depends on the second moment of $Y$ (as $\left\Vert \theta^{\star}\right\Vert _{\Sigma}^{2}\leq \esp Y^2$) and decreases with a slow rate $1/\sqrt{n}$. As seen in Section~\ref{sec:bias}, the imputation bias is upper-bounded by the ridge bias  with penalization parameter $\lambda_{\imp}$, which is controlled in high dimension  for low-rank data (see examples in Section~\ref{sec:bias_and_Sigma}).  

The bias $(i)$ due to the initial condition is the most critical. Indeed,  $\mathrm{Tr}(\Sigma)= \esp[\| X \|_2^2]$  is likely to increase with $d$, e.g., under Assumption~\ref{ass:2moment_sup}, $\mathrm{Tr}(\Sigma)\leq d L^2$. Besides, the starting point  $\theta_0$ may be far from $\theta_{\imp}^\star$. Fortunately, Lemma~\ref{lem:norm_control} establishes some properties of  $\theta_{\imp}^\star$. 

\begin{lemma}\label{lem:norm_control}
    Under Assumptions~\ref{ass:MCAR} and \ref{ass:2moment_inf}, let $V$ be the covariance matrix of $P$ defined in \cref{prop:regularization_implicite}. If $V$ is invertible, then 
    \begin{equation}
            \left\Vert \theta_{\imp}^\star\right\Vert _{2}^{2}\leq \frac{B_{\imp}(\F)}{\ell^2\lambda_{\min}(V)}. 
        \end{equation}
In particular, under \cref{ass:bernoulli_model},
\begin{equation}
            \left\Vert \theta_{\imp}^\star\right\Vert _{2}^{2}\leq \frac{B_{\imp}(\F)}{\ell^2\rho(1-\rho)}. 
        \end{equation}
\end{lemma}
Lemma~\ref{lem:norm_control} controls the norm of the optimal predictor $\theta_{\imp}^\star$  by the imputation bias: if  the imputation bias is small, then the optimal predictor on zero-imputed data is of low norm. According to Section~\ref{sec:bias}, this holds in particular for high-dimensional settings.  
Thus, choosing $\theta_0=0$ permits us  to exploit the upper bound provided by Lemma \ref{lem:norm_control} in Theorem \ref{thm:SGD_bound}. With such an initialization, the bias due to this initial condition is upper bounded by $\frac{\kappa \mathrm{Tr}(\Sigma)}{\sqrt{n}}\| \theta^{\star}_{\imp}\|_{2}^{2}$. 
Intuitively, as $\theta_{\imp}^\star$ is in an $\ell^2$-ball of small radius, choosing $\theta_0$ within that ball, e.g. $\theta_0=0$ is a good choice. 



   
Taking into account \Cref{lem:norm_control}, \Cref{prop:SGD_bound_bernoulli}  establishes our final upper bound on SGD on zero-imputed data.


\begin{propo}\label{prop:SGD_bound_bernoulli}
Under Assumptions~\ref{ass:bernoulli_model}, \ref{ass:2moment_sup}, \ref{ass:2moment_inf} and \ref{ass:SGD}, the predictor $\bar{\theta}_{\imp,n}$ resulting from the SGD strategy, defined in Section~\ref{subsec:algo}, with starting point $\theta_0=0$ and learning rate  $\gamma = \frac{1}{d\kappa L^2\sqrt{n}}$, satisfies 
\begin{align*}
   \esp\left[R_{\imp}\left(\bar{\theta}_{\imp,n}\right)\right]- R^{\star}(\F) 
  \lesssim &\left(\frac{L^2}{\ell^2}\frac{\kappa d}{\rho(1-\rho)\sqrt{n}}+1\right)B_{\imp}(\F) + \frac{\sigma^2+\left\Vert \theta^{\star}\right\Vert _{\Sigma}^{2}}{\sqrt{n}}  .  
\end{align*}
\end{propo}
In this upper bound, the first term encapsulates the imputation bias and the one due to the initial condition, whilst the second one corresponds to the variance of the training procedure. 
As soon as $d \gg \frac{\ell^2}{L^2}\frac{\rho(1-\rho)\sqrt{n}}{\kappa}$ then the imputation bias is negligible compared to that of the initial condition. 


\subsection{Examples}

According to \cref{ex:low_rank,ex:noised_low_rank}, $B_{\imp}(\F)$ decreases with the dimension, provided that $\Sigma$ or $\beta$ are structured. Strikingly, Corollary~\ref{cor:SGD_upp_bound} highlights cases where the upper bound of \cref{prop:SGD_bound_bernoulli} is actually dimension-free.

 
\begin{corollary}
\label{cor:SGD_upp_bound}
Suppose that assumptions of \cref{prop:SGD_bound_bernoulli} hold. Recall that $\lambda_1 \geq \hdots \geq \lambda_d$ are the eigenvalues of $\Sigma$ associated with the eigenvectors $v_1, \hdots, v_d$. 
\begin{enumerate}[leftmargin=*]
    \item[$(i)$]  \emph{(\cref{ex:low_rank} - Low-rank $\Sigma$).} If $\Sigma$ has a low rank $r\ll d$ and equal non-zero singular values, then 
    \begin{equation*}
      \esp\left[R_{\imp}\left(\bar{\theta}_{\imp,n}\right)\right]-R^{\star}(\F)  \lesssim \frac{L^2}{\ell^2}\left(\frac{L^2}{\ell^2}\frac{\kappa}{\rho\sqrt{n}}+\frac{1-\rho}{d}\right)\frac{ r \left\Vert \theta^{\star}\right\Vert _{\Sigma}^{2}}{\rho} + \frac{\sigma^2}{\sqrt{n}}.  
    \end{equation*}
    \item[$(ii)$]  \emph{(\cref{ex:noised_low_rank} - Spiked model).}
    If $\Sigma= \Sigma_{\leq r}+ \Sigma_{> r} $ with  $\Sigma_{> r} \preceq \ell^2\eta I$, $\Sigma_{\leq r}$ has a low rank $r\ll d$ with equal non-zero singular values, and the projection of $\theta^{\star}$ on the range of $\Sigma_{> r}$ satisfies $\theta^{\star}_{>r}=0$, then 
    \begin{equation*}
     \esp\left[R_{\imp}\left(\bar{\theta}_{\imp,n}\right)\right]-R^{\star}(\F) \lesssim \frac{L^2}{\ell^2}\left(\frac{L^2}{\ell^2}\frac{\kappa}{\rho\sqrt{n}}+\frac{1-\rho}{d}\right)\frac{ r \left\Vert \theta^{\star}\right\Vert _{\Sigma}^{2}}{\rho (1-\eta)}+ \frac{\sigma^2}{\sqrt{n}}.   
    \end{equation*}
\end{enumerate}

\end{corollary}

\Cref{cor:SGD_upp_bound} establishes upper bounds on the risk of SGD applied on zero-imputed data, for some particular structures on $\Sigma$ and $\theta^{\star}$. These bounds take into account the statistical error as well as the optimization one, and are expressed as function of $d$ and $n$.  
Since $\left\Vert \theta^{\star}\right\Vert _{\Sigma}^{2}$ is upper bounded by $\mathbb{E}Y^2$ (a dimension-free term), the risks in \Cref{cor:SGD_upp_bound} can also be upper bounded by dimension-free quantities, provided   $d>\frac{\ell^2}{L^2}\frac{\rho(1-\rho)\sqrt{n}}{\kappa}$.

Besides, \Cref{cor:SGD_upp_bound} shows that, for $d \gg \frac{\ell^2}{L^2}\frac{\rho(1-\rho)\sqrt{n}}{\kappa}$, the imputation bias is negligible with respect to the stochastic error of SGD. Therefore, for structured problems in high-dimensional settings for which $d \gg \frac{\ell^2}{L^2}\frac{\rho(1-\rho)\sqrt{n}}{\kappa}$, the zero-imputation strategy is consistent, with a slow rate of order $1/\sqrt{n}$.

\begin{remark}[Discussion about slow rates]
An important limitation of coupling naive imputation with SGD is that fast convergence rates cannot be reached. Indeed, in large dimensions, the classical fast rate is given by $\mathrm{Tr}(\Sigma(\Sigma+\lambda I)^{-1})/n$ with $\lambda$ the penalization hyper-parameter.  The quantity $\mathrm{Tr}(\Sigma(\Sigma+\lambda I)^{-1})$, often called degrees of freedom, can be negligible w.r.t.\ $d$ (for instance when $\Sigma$ has a fast eigenvalue decay). However, when working with an imputed dataset, the covariance matrix of the data is not $\Sigma$ anymore, but $\Sigma_{\imp} =\esp X_{\imp}X_{\imp}^{\top}$.  
Therefore, in the case of \cref{ass:bernoulli_model} (Ho-MCAR), 
all the eigenvalues of $\Sigma_{\imp}$ are larger than $\rho(1-\rho)$ (preventing the eigenvalues decay obtained when working with complete inputs). By concavity of the degrees of freedom (on positive semi-definite matrix),  we can show that 
$\mathrm{Tr}(\Sigma_{\imp}(\Sigma_{\imp}+\lambda I)^{-1}) \geq \frac{d\rho(1-\rho)}{1+\lambda}$, 
hindering traditional fast rates.    
\end{remark}

\paragraph{Link with dropout} Dropout is a classical regularization technique used in deep learning, consisting in randomly discarding some neurons at each SGD iteration~\citep{srivastava2014dropout}. Regularization properties of dropout have attracted a lot of attention \citep[e.g., ][]{gal2016theoretically}. Interestingly, setting a neuron to 0 on the input layer is equivalent to masking the corresponding feature.  
Running SGD (as in \Cref{sec:sgd}) on a stream of zero-imputed data is thus equivalent to training a neural network with no hidden layer, a single output neuron, and dropout on the input layer. Our theoretical analysis describes the implicit regularization impact of dropout in that very particular case. Interestingly, this can also be applied to the fine-tuning of the last layer of any regression network structure.

\section{Numerical experiments}\label{sec:experiments}

\paragraph{Data simulation}
We generate $n=500$ complete input data according to a normal distribution with two different covariance structures. First, in the \textbf{low-rank} setting (Ex.~\ref{ex:low_rank} and \ref{ex:low_rank2}), the output is formed as 
$Y=\beta^\top Z+\epsilon$, with $\beta\in\R^r$, $Z\sim\mathcal{N}(0,I_r)$ and $\epsilon\sim\mathcal{N}(0,2)$, and the inputs are given by $X=AZ+\mu$, with a full rank matrix $A\in \R^{d\times r}$ and a mean vector $\mu\in\R^d$. 
Note that the dimension $d$ varies in the experiments, while $r=5$ is kept fixed. Besides, the full model can be rewritten as  $Y=X^\top \theta^\star +\epsilon$ with $\theta^\star = (A^{\dagger})^\top \beta$ where $A^{\dagger}$ is the Moore-Penrose inverse of $A$. Secondly, in the \textbf{spiked model} (Ex.~\ref{ex:noised_low_rank}), the input and the output are decomposed as $X = (X_1,X_2)\in\mathbb{R}^{d/2}\times \mathbb{R}^{d/2}$ and $Y = Y_1+Y_2$, where  $(X_1, Y_1)$ is generated according to the low-rank model above and $(X_2, Y_2)$ is given by a linear model $Y_2= \theta_2^\top X_2$ and $X_2\sim \mathcal{N}(0,I_{d/2})$,  choosing $\Vert \theta_2\Vert=0.2$.

Two missing data scenarios, with a proportion $\rho$ of observed entries equal to $50\%$, are simulated according to (i) the {Ho-MCAR} setting (\Cref{ass:bernoulli_model}); 
and to 
(ii) the {self-masking MNAR} setting, which departs significantly from the MCAR case as the presence of missing data depends on the underlying value itself. More precisely, set $\alpha \in \R^d$ such that,  for all $j\in[d]$, 
$\P(P_j=1|X)=(1+e^{-\alpha_j X_j})^{-1}$ and $\esp [P_j] = 0.5$ ($50\%$ of missing data on average per components). 


\paragraph{Regressors}  For two-step strategies, different imputers are  combined with different regressors. The considered imputers are:  
the zero imputation method (\textbf{0-imp}) complying with the theoretical analysis developed in this paper, 
the optimal imputation by a constant for each input variable (\textbf{Opti-imp}), obtained by training a linear model on the augmented data $(P\odot X,P)$  \citep[see][Proposition 3.1]{le2020linear}, and single imputation by chained equations (\textbf{ICE}, \citep{van2011mice})\footnote{ \texttt{IterativeImputer} in scikit-learn  \citep{scikit-learn}.}.
The subsequent regressors, implemented in scikit-learn \citep{scikit-learn}, are either the averaged SGD (\textbf{SGD}, package \texttt{SGDRegressor}) with $\theta_0=0$ and $\gamma = (d\sqrt{n})^{-1}$ (see  \Cref{prop:SGD_bound_bernoulli}, or the ridge regressor 
(with a leave-one-out cross-validation, package \texttt{ridge}). 
Two specific methods that do not resort to prior imputation are also assessed: a pattern-by-pattern regressor \citep{le2020linear,ayme2022near} (\textbf{Pat-by-Pat})
and a neural network architecture (\textbf{NeuMiss}) \citep{lemorvan:hal-02888867} specifically designed to handle missing data in linear prediction.

\paragraph{Numerical results} In Figure \ref{fig:excess_risk_XP} (a) and (b), we consider Ho-MCAR patterns with Gaussian inputs with resp.\ a low-rank and spiked covariance matrix.
The 2-step strategies perform remarkably well, with the ICE imputer on the top of the podium, highly appropriate to the type of data (MCAR Gaussian) in play. Nonetheless, the naive imputation by zero remains competitive in terms of predictive performance and is computationally efficient, with a complexity of $O(nd)$, especially compared to ICE, whose complexity is of order $n^2d^3$. 
Regarding Figure \ref{fig:excess_risk_XP} (b), we note that ridge regression outperforms SGD for large $d$. Note that, in the regime where $d \geq  \sqrt{n}$, the imputation bias is negligible w.r.t.\ to the method bias, the latter being lower in the case of ridge regression. 
This highlights the benefit of explicit ridge regularization (with a tuned hyperparameter) over the implicit regularization induced by the imputation.

\begin{figure}[h]
    \centering
    \includegraphics[width= 0.9\linewidth]{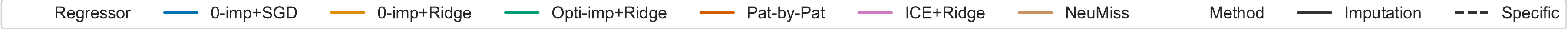}
    \begin{tabular}{ccc}
        \includegraphics[width=0.27\linewidth]{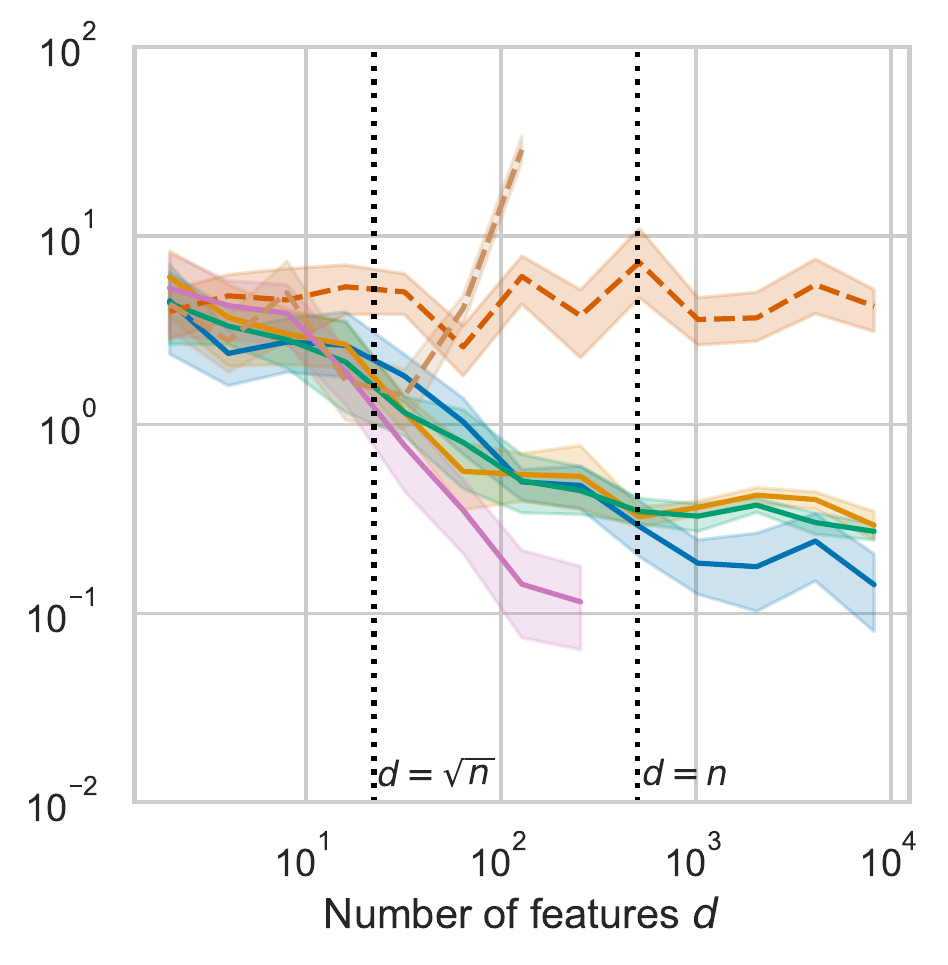}&
        \includegraphics[width=0.27\linewidth]{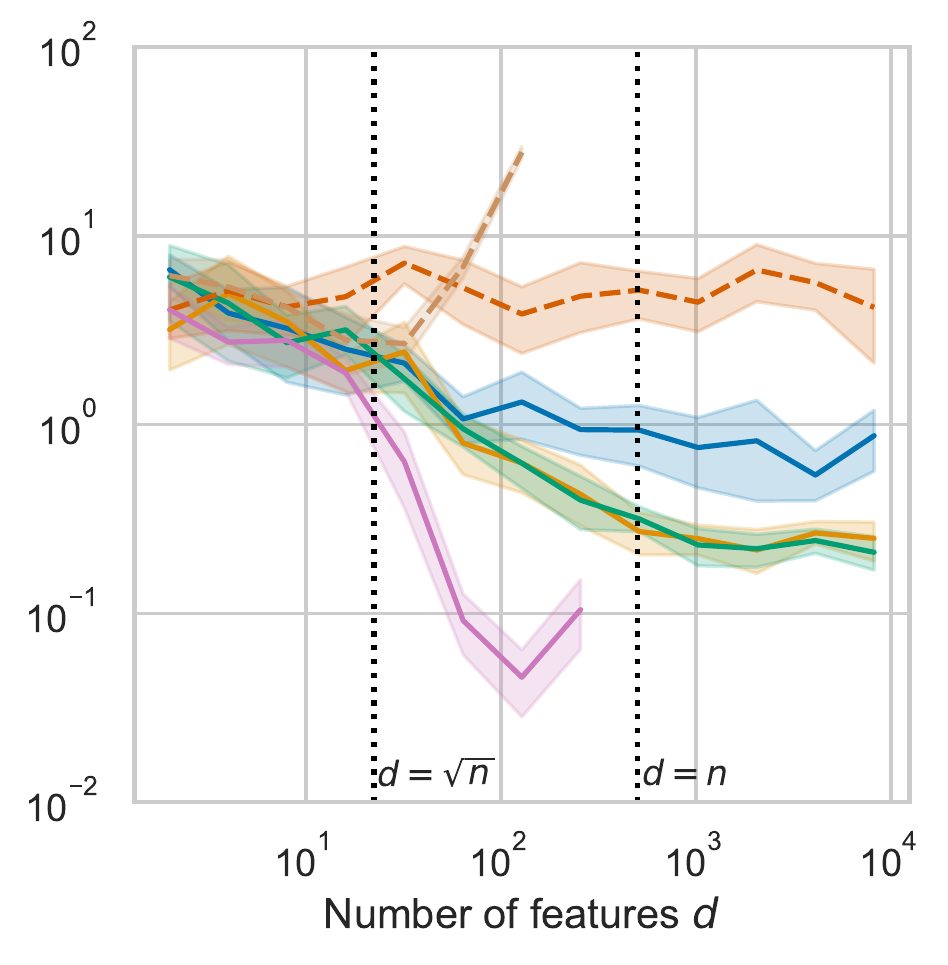}&
        \includegraphics[width=0.27\linewidth]{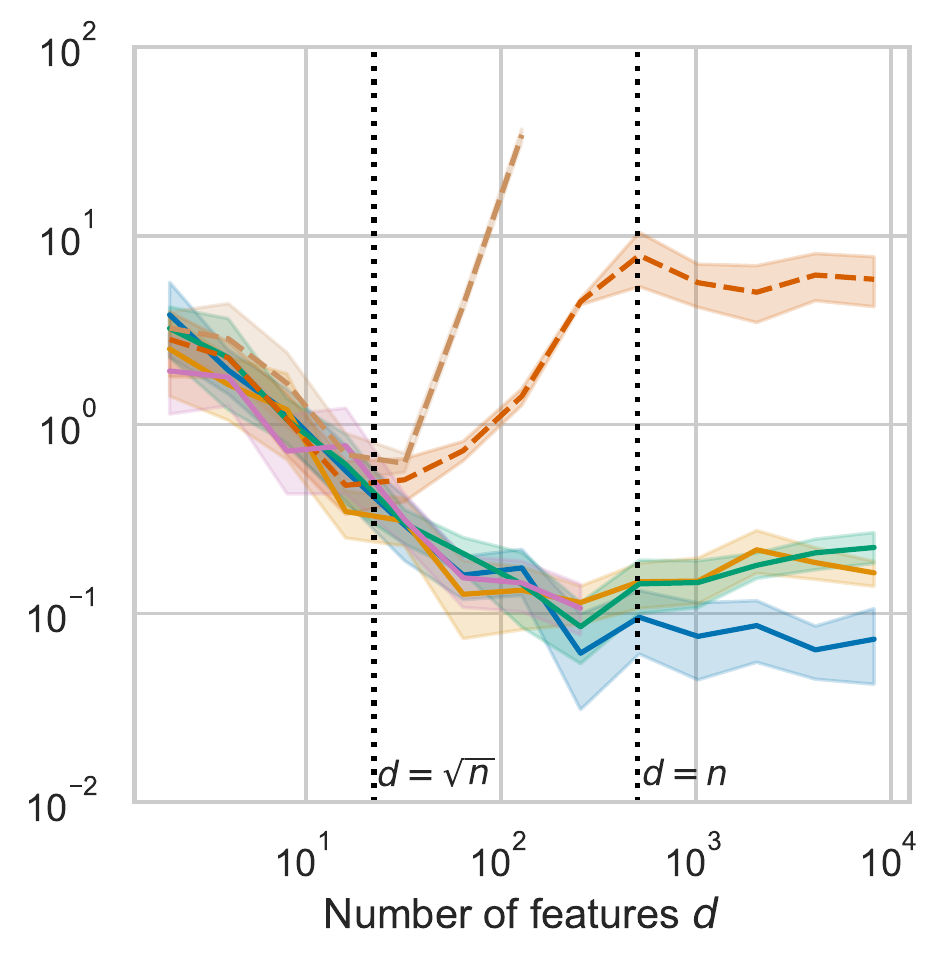} \\
        {(a) Ho-MCAR } & {(b) Ho-MCAR } & {(c) Self-Masked } \\
        + Low-rank model & + Spiked model & + Low-rank model
    \end{tabular}
    \caption{\label{fig:excess_risk_XP} Risk w.r.t.~the input dimension (evaluated on $10^4$ test samples) when $50\%$ of the input data is missing. The $y$-axis corresponds to $R_{\mathrm{mis}}(f)-R^\star= \mathbb{E}\left[ \left( Y-f(X_{\imp},P) \right)^2 \right]-\sigma^2$. The averaged risk is depicted over 10 repetitions within a 95\% confidence interval. }
\end{figure}

In practice, missing data are not always of the Ho-MCAR type, we compare therefore the different algorithms on self-masked data.
In Figure \ref{fig:excess_risk_XP} (c), we note that specific methods remain competitive for larger $d$ compared to MCAR settings. This was to be expected since those methods were designed to handle complex missing not at random (MNAR) data. However, they still suffer from the curse of dimensionality and turns out to be inefficient in large dimension, compared to all two-step strategies.

\section{Discussion and conclusion}

In this paper, we study the impact of zero imputation in  high-dimensional linear models. 
We demystify this widespread technique, by exposing its implicit regularization mechanism when dealing with MCAR data. 
We prove that, in high-dimensional regimes, the induced bias is similar to that of ridge regression, commonly accepted by practitioners. 
By providing generalization bounds on SGD trained on zero-imputed data, we establish that such two-step procedures are statistically sound, while being computationally appealing. 

Theoretical results remain to be established beyond the MCAR case, to properly analyze and compare the different strategies for dealing with missing data in MNAR settings (see Figure~\ref{fig:excess_risk_XP} (c)). 
Extending our results to a broader class of functions (escaping linear functions) or even in a classification framework, would be valuable to fully understand the properties of imputation.

\bibliography{references}

\begin{thebibliography}{27}
\providecommand{\natexlab}[1]{#1}
\providecommand{\url}[1]{\texttt{#1}}
\expandafter\ifx\csname urlstyle\endcsname\relax
  \providecommand{\doi}[1]{doi: #1}\else
  \providecommand{\doi}{doi: \begingroup \urlstyle{rm}\Url}\fi

\bibitem[Agarwal et~al.(2019)Agarwal, Shah, Shen, and
  Song]{agarwal2019robustness}
Anish Agarwal, Devavrat Shah, Dennis Shen, and Dogyoon Song.
\newblock On robustness of principal component regression.
\newblock \emph{Advances in Neural Information Processing Systems}, 32, 2019.

\bibitem[Ayme et~al.(2022)Ayme, Boyer, Dieuleveut, and Scornet]{ayme2022near}
Alexis Ayme, Claire Boyer, Aymeric Dieuleveut, and Erwan Scornet.
\newblock Near-optimal rate of consistency for linear models with missing
  values.
\newblock In \emph{International Conference on Machine Learning}, pages
  1211--1243. PMLR, 2022.

\bibitem[Bach and Moulines(2013)]{bach2013non}
Francis Bach and Eric Moulines.
\newblock Non-strongly-convex smooth stochastic approximation with convergence
  rate o (1/n).
\newblock \emph{Advances in neural information processing systems}, 26, 2013.

\bibitem[Bartlett et~al.(2020)Bartlett, Long, Lugosi, and
  Tsigler]{bartlett2020benign}
Peter~L Bartlett, Philip~M Long, G{\'a}bor Lugosi, and Alexander Tsigler.
\newblock Benign overfitting in linear regression.
\newblock \emph{Proceedings of the National Academy of Sciences}, 117\penalty0
  (48):\penalty0 30063--30070, 2020.

\bibitem[Caponnetto and De~Vito(2007)]{caponnetto2007optimal}
Andrea Caponnetto and Ernesto De~Vito.
\newblock Optimal rates for the regularized least-squares algorithm.
\newblock \emph{Foundations of Computational Mathematics}, 7\penalty0
  (3):\penalty0 331--368, 2007.

\bibitem[Chizat and Bach(2020)]{chizat2020implicit}
Lenaic Chizat and Francis Bach.
\newblock Implicit bias of gradient descent for wide two-layer neural networks
  trained with the logistic loss.
\newblock In \emph{Conference on Learning Theory}, pages 1305--1338. PMLR,
  2020.

\bibitem[Dieuleveut et~al.(2017)Dieuleveut, Flammarion, and
  Bach]{dieuleveut2017harder}
Aymeric Dieuleveut, Nicolas Flammarion, and Francis Bach.
\newblock Harder, better, faster, stronger convergence rates for least-squares
  regression.
\newblock \emph{The Journal of Machine Learning Research}, 18\penalty0
  (1):\penalty0 3520--3570, 2017.

\bibitem[Gal and Ghahramani(2016)]{gal2016theoretically}
Yarin Gal and Zoubin Ghahramani.
\newblock A theoretically grounded application of dropout in recurrent neural
  networks.
\newblock \emph{Advances in neural information processing systems}, 29, 2016.

\bibitem[Hsu et~al.(2012)Hsu, Kakade, and Zhang]{hsu2012random}
Daniel Hsu, Sham~M Kakade, and Tong Zhang.
\newblock Random design analysis of ridge regression.
\newblock In \emph{Conference on learning theory}, pages 9--1. JMLR Workshop
  and Conference Proceedings, 2012.

\bibitem[Ipsen et~al.(2022)Ipsen, Mattei, and Frellsen]{ipsen2022deal}
Niels~Bruun Ipsen, Pierre-Alexandre Mattei, and Jes Frellsen.
\newblock How to deal with missing data in supervised deep learning?
\newblock In \emph{ICLR 2022-10th International Conference on Learning
  Representations}, 2022.

\bibitem[Johnstone(2001)]{johnstone2001}
Iain~M. Johnstone.
\newblock {On the distribution of the largest eigenvalue in principal
  components analysis}.
\newblock \emph{The Annals of Statistics}, 29\penalty0 (2):\penalty0 295 --
  327, 2001.
\newblock \doi{10.1214/aos/1009210544}.
\newblock URL \url{https://doi.org/10.1214/aos/1009210544}.

\bibitem[Josse et~al.(2019)Josse, Prost, Scornet, and
  Varoquaux]{josse2019consistency}
Julie Josse, Nicolas Prost, Erwan Scornet, and Ga{\"e}l Varoquaux.
\newblock On the consistency of supervised learning with missing values.
\newblock \emph{arXiv preprint arXiv:1902.06931}, 2019.

\bibitem[Le~Morvan et~al.(2020{\natexlab{a}})Le~Morvan, Josse, Moreau, Scornet,
  and Varoquaux]{lemorvan:hal-02888867}
Marine Le~Morvan, Julie Josse, Thomas Moreau, Erwan Scornet, and Ga{\"e}l
  Varoquaux.
\newblock {NeuMiss networks: differentiable programming for supervised learning
  with missing values}.
\newblock In \emph{{NeurIPS 2020 - 34th Conference on Neural Information
  Processing Systems}}, Vancouver / Virtual, Canada, December
  2020{\natexlab{a}}.
\newblock URL \url{https://hal.archives-ouvertes.fr/hal-02888867}.

\bibitem[Le~Morvan et~al.(2020{\natexlab{b}})Le~Morvan, Prost, Josse, Scornet,
  and Varoquaux]{le2020linear}
Marine Le~Morvan, Nicolas Prost, Julie Josse, Erwan Scornet, and Ga{\"e}l
  Varoquaux.
\newblock Linear predictor on linearly-generated data with missing values: non
  consistency and solutions.
\newblock In \emph{International Conference on Artificial Intelligence and
  Statistics}, pages 3165--3174. PMLR, 2020{\natexlab{b}}.

\bibitem[Le~Morvan et~al.(2021)Le~Morvan, Josse, Scornet, and
  Varoquaux]{le2021sa}
Marine Le~Morvan, Julie Josse, Erwan Scornet, and Ga{\"e}l Varoquaux.
\newblock What’sa good imputation to predict with missing values?
\newblock \emph{Advances in Neural Information Processing Systems},
  34:\penalty0 11530--11540, 2021.

\bibitem[Mourtada(2019)]{mourtada2019exact}
Jaouad Mourtada.
\newblock Exact minimax risk for linear least squares, and the lower tail of
  sample covariance matrices.
\newblock \emph{arXiv preprint arXiv:1912.10754}, 2019.

\bibitem[Mourtada and Rosasco(2022)]{mourtada2022elementary}
Jaouad Mourtada and Lorenzo Rosasco.
\newblock An elementary analysis of ridge regression with random design.
\newblock \emph{arXiv preprint arXiv:2203.08564}, 2022.

\bibitem[Muzellec et~al.(2020)Muzellec, Josse, Boyer, and
  Cuturi]{muzellec2020missing}
Boris Muzellec, Julie Josse, Claire Boyer, and Marco Cuturi.
\newblock Missing data imputation using optimal transport.
\newblock In \emph{International Conference on Machine Learning}, pages
  7130--7140. PMLR, 2020.

\bibitem[Pedregosa et~al.(2011)Pedregosa, Varoquaux, Gramfort, Michel, Thirion,
  Grisel, Blondel, Prettenhofer, Weiss, Dubourg, Vanderplas, Passos,
  Cournapeau, Brucher, Perrot, and Duchesnay]{scikit-learn}
F.~Pedregosa, G.~Varoquaux, A.~Gramfort, V.~Michel, B.~Thirion, O.~Grisel,
  M.~Blondel, P.~Prettenhofer, R.~Weiss, V.~Dubourg, J.~Vanderplas, A.~Passos,
  D.~Cournapeau, M.~Brucher, M.~Perrot, and E.~Duchesnay.
\newblock Scikit-learn: Machine learning in {P}ython.
\newblock \emph{Journal of Machine Learning Research}, 12:\penalty0 2825--2830,
  2011.

\bibitem[Pesme et~al.(2021)Pesme, Pillaud-Vivien, and
  Flammarion]{pesme2021implicit}
Scott Pesme, Loucas Pillaud-Vivien, and Nicolas Flammarion.
\newblock Implicit bias of sgd for diagonal linear networks: a provable benefit
  of stochasticity.
\newblock \emph{Advances in Neural Information Processing Systems},
  34:\penalty0 29218--29230, 2021.

\bibitem[Polyak and Juditsky(1992)]{polyak1992acceleration}
Boris~T Polyak and Anatoli~B Juditsky.
\newblock Acceleration of stochastic approximation by averaging.
\newblock \emph{SIAM journal on control and optimization}, 30\penalty0
  (4):\penalty0 838--855, 1992.

\bibitem[Rubin(1976)]{RUBIN76}
DONALD~B. Rubin.
\newblock {Inference and missing data}.
\newblock \emph{Biometrika}, 63\penalty0 (3):\penalty0 581--592, 12 1976.
\newblock ISSN 0006-3444.
\newblock \doi{10.1093/biomet/63.3.581}.
\newblock URL \url{https://doi.org/10.1093/biomet/63.3.581}.

\bibitem[Sridharan et~al.(2008)Sridharan, Shalev-Shwartz, and
  Srebro]{sridharan2008fast}
Karthik Sridharan, Shai Shalev-Shwartz, and Nathan Srebro.
\newblock Fast rates for regularized objectives.
\newblock \emph{Advances in neural information processing systems}, 21, 2008.

\bibitem[Srivastava et~al.(2014)Srivastava, Hinton, Krizhevsky, Sutskever, and
  Salakhutdinov]{srivastava2014dropout}
Nitish Srivastava, Geoffrey Hinton, Alex Krizhevsky, Ilya Sutskever, and Ruslan
  Salakhutdinov.
\newblock Dropout: a simple way to prevent neural networks from overfitting.
\newblock \emph{The journal of machine learning research}, 15\penalty0
  (1):\penalty0 1929--1958, 2014.

\bibitem[Van~Buuren and Groothuis-Oudshoorn(2011)]{van2011mice}
Stef Van~Buuren and Karin Groothuis-Oudshoorn.
\newblock mice: Multivariate imputation by chained equations in r.
\newblock \emph{Journal of statistical software}, 45:\penalty0 1--67, 2011.

\bibitem[Yao et~al.(2007)Yao, Rosasco, and Caponnetto]{yao2007early}
Yuan Yao, Lorenzo Rosasco, and Andrea Caponnetto.
\newblock On early stopping in gradient descent learning.
\newblock \emph{Constructive Approximation}, 26\penalty0 (2):\penalty0
  289--315, 2007.

\bibitem[Yoon et~al.(2018)Yoon, Jordon, and Schaar]{yoon2018gain}
Jinsung Yoon, James Jordon, and Mihaela Schaar.
\newblock Gain: Missing data imputation using generative adversarial nets.
\newblock In \emph{International conference on machine learning}, pages
  5689--5698. PMLR, 2018.

\end{thebibliography}
\bibliographystyle{plainnat}

\newpage
\appendix

\section{Notations}

\label{app:notations}
For two vectors (or matrices) $a,b$, we denote by $a\odot b$ the Hadamard product (or component-wise product). $[n]=\left\{1,2,...,n\right\}$. For two symmetric matrices $A$ and $B$, $A\preceq B$ means that $B-A$ is positive semi-definite. The symbol $\lesssim$ denotes the inequality up to a universal constant.  Table~\ref{tab:notations} summarizes the  notations used throughout the paper.

\begin{table}[h!]
    \centering
        \caption{Notations}
            \label{tab:notations}

    \begin{tabular}{l l} \toprule
         $P$ & Mask\\ 
         $\F$ &  Set of linear functions \\ 
         $B_{\imp}$ & Imputation bias   \\ 
          $\Sigma$ & $\esp XX^\top$  \\
          $\lambda_j$ & eigenvalues of $\Sigma$  \\
          $v_j$ & eigendirections of $\Sigma$  \\
          $\Sigma_P$ & $\esp PP^\top$  \\
         $L^2$& the largest second moments $\mathrm{max}_j \esp X_j^2$ (\cref{ass:2moment_sup})\\
         $\ell^2$ &the smallest second moments $\mathrm{min}_j \esp X_j^2$ (\cref{ass:2moment_inf})\\
         $\theta^\star$ &Best linear predictor on complete data \\
         $\theta_{\imp}^\star$&Best linear predictor on imputed data\\
         $r$& Rank of $\Sigma$\\
         $\rho_j$ & Theoretical proportion of observed entries \\
         & for the $j$-th variable in a MCAR setting \\
         $V$& Covariance matrix associated to the missing patterns \\
         $C$&  Covariance matrix $V$ renormalized by $(\rho_j)_j$ defined in \eqref{eq:corrMask}\\
         $\kappa$ &  Kurtosis of the input $X$\\
         \bottomrule
    \end{tabular}
    \end{table}

\section{Proof of the main results}

\subsection{Proof of \Cref{lem:RmisVsR}}
The proof is based on the definition of the conditional
expectation, and given that 
\begin{align*}
R^{\star} & =\esp\left[\left(Y-\esp\left[Y|X\right]\right)^{2}\right].
\end{align*}
Note that $\esp\left[Y|X,P\right]=\esp\left[f^{\star}(X)+\epsilon|X,P\right]=\esp\left[f^{\star}(X)|X,P\right]=f^{\star}(X)$ (by independence of $\epsilon$ and $P$).
Therefore, 
\begin{align*}
R^{\star} & =\esp\left[\left(Y-f^{\star}(X)\right)^{2}\right]\\
 & \leq\esp\left[\left(Y-\esp\left[Y|X,P\right]\right)^{2}\right]\\
 & \leq\esp\left[\left(Y-\esp\left[Y|X_{\imp},P\right]\right)^{2}\right]\\
 & \leq R_{\mathrm{mis}}^{\star},
\end{align*}
using that $\esp\left[Y|X_{\imp},P\right]$ is a measurable function
of $\left(X,P\right)$.

\subsection{Preliminary lemmas}
\paragraph{Notation} Let $X_a$ be a random variable of law $\mathcal{L}_a$ (a modified version of the law of the underlying input $X$) on $\R^d$, and for $f\in\F$ define
\begin{equation*}
    R_a(f)=\esp\left[\left(Y-f(X_a)\right)^{2}\right],
\end{equation*}
the associate risk. The Bayes risk is given by
\begin{equation*}
    R_a^\star(\F)=\inf_{f\in \F}\esp\left[\left(Y-f(X_a)\right)^{2}\right],
\end{equation*}
if the infimum is reached, we denote by $f_a^\star\in\arg\min_{f\in \F}R_a(f)$. The discrepancy between both risks, involving either the modified input $X_a$ or the initial input $X$, can be measured through the following bias: 
\begin{equation*}
    B_{a}=R_a^\star(\F)-R^\star(\F).
\end{equation*}

\paragraph{General decomposition}
The idea of the next lemma is to compare $R_a(f)$ with the true risk $R(f)$.
\begin{lemma}\label{lem:Bias_variance_X'}
If $(X_a \perp\!\!\!\perp Y)|X$, then, for all $\theta\in\R^{d}$,
\[
R_a\left(f_\theta\right)=R\left(g_{\theta}\right)+\left\Vert\theta\right\Vert _{\Gamma}^{2},
\]
where $g_{\theta}(X)=\theta^\top\esp\left[X_a|X\right]$ and $\Gamma=\esp\left[(X_a-\esp\left[X_a|X\right])(X_a-\esp\left[X_a|X\right])^\top\right] $ the integrated conditional covariance matrix. In consequence, if there exists an invertible linear application $H$ such that, $\esp\left[X_a|X\right] = H^{-1}X$, then 
\begin{itemize}
    \item For all $\theta\in\R^d$, $g_{\theta}$ is a linear function 
    and 
    \begin{equation}\label{eq:BVA_risk_star}
R_a^\star(\F)=\inf_{\theta\in\R^d} \left\{R\left(f_{\theta}\right)+\left\Vert\theta\right\Vert _{H^\top \Gamma H}^{2}\right\}.
    \end{equation}

    \item If $\lambda_{\rm{max}}(H \Gamma H^\top)\leq \Lambda $, then 
    \begin{equation}\label{eq:BVA_bias}
        B_a(\F)\leq \inf_{\theta \in \R^d} \left\{R(f_\theta)+ \Lambda \left\Vert\theta\right\Vert _{2}^{2}\right\}= B_{\rm{ridge},\Lambda}.
    \end{equation}
    \item If $\lambda_{\rm{min}}(\Gamma)\geq \mu>0 $, then 
    \begin{equation}\label{eq:BVA_norm}
        \left\Vert\theta_a^\star\right\Vert _{2}^{2}\leq \frac{B_a(\F)}{\mu} . 
    \end{equation}
\end{itemize}
\end{lemma}
\begin{remark}
\cref{eq:BVA_norm} is crucial because a bound on the bias $B_a(\F)$ actually gives a bound for $\left\Vert\theta_a^\star\right\Vert _{2}^{2}$ too. 
This will be of particular interest for  \cref{thm:SGD_bound}.  
\end{remark}

\begin{proof}
\begin{align*}
R_a\left(f_{\theta}\right) & =\esp\left[\left(Y-\theta^{\top}X_a\right)^{2}\right]\\
 & =\esp\left[\esp\left[\left(Y-\esp\left[\theta^{\top}X_a|X\right]+\esp\left[\theta^{\top}X_a|X\right]-\theta^{\top}X_a\right)^{2}\Big| X\right]\right]\\
 & =\esp\left[\left(Y-\esp\left[\theta^{\top}X_a|X\right]\right)^{2}\right]+\esp\left[\esp\left[\left(\esp\left[\theta^{\top}X_a|X\right]-\theta^{\top}X_a\right)^{2} \Big| X\right]\right]\\
 & =\esp\left[\left(Y-g_\theta(X)\right)^{2}\right]+\esp\left[\esp\left[\left(\esp\left[\theta^{\top}X_a|X\right]-\theta^{\top}X_a\right)^{2} \Big| X\right]\right] \\
 & =R(g_\theta) +\esp\left[\esp\left[\left(\esp\left[\theta^{\top}X_a|X\right]-\theta^{\top}X_a\right)^{2} \Big| X\right]\right].
\end{align*}
since $\esp\left[\esp\left[\theta^{\top}X_a|X\right]-\theta^{\top}X_a|X\right]=0$.
Furthermore, 
\begin{align*}
\esp\left[\esp\left[\left(\esp\left[\theta^{\top}X_a|Z\right]-\theta^{\top}X_a\right)^{2}|X\right]\right] & =\theta^{\top}\esp\left[\left(\esp\left[X_a|X\right]-X_a\right)\left(\esp\left[X_a|X\right]-X_a\right)^{\top}\right]\theta\\
 & =\esp\left[\theta^{\top}\esp\left[\left(\esp\left[X_a|X\right]-X_a\right)\left(\esp\left[X_a|X\right]-X_a\right)^{\top}|X\right]\theta\right]\\
 & =\esp\left[\left\Vert \theta\right\Vert _{\esp\left[\left(\esp\left[X_a|X\right]-X_a\right)\left(\esp\left[X_a|X\right]-X_a\right)^{\top}|X\right]}^{2}\right]\\
 & =\esp\left[\left\Vert \theta\right\Vert _{\Gamma}^{2}\right].
\end{align*}
Finally,
\begin{align*}
R_a\left(f_{\theta}\right) & = R(g_\theta) + \|\theta\|_\Gamma^2.
\end{align*}

Assume that an invertible matrix $H$ exists such that $g_{\theta}(X)=\theta^\top H^{-1}X$, thus $g_\theta$ is a linear function. Equation  \eqref{eq:BVA_risk_star} is then obtained by using a change of variable: $\theta'=(H^{-1})^\top \theta= (H^\top)^{-1}\theta$ and $\theta=H^\top \theta'$. Thus, we have $g_{\theta'}(X)=\theta^\top X=f_{\theta}(X)$ and 
\begin{align*}
R_a\left(f_{\theta'}\right) & = R(f_\theta) + \|H^\top \theta'\|_\Gamma^2\\
& = R(f_\theta) + \| \theta'\|_{H\Gamma H^\top}^2.
\end{align*}

Then using $H \Gamma H^\top\preceq \Lambda I$ proves \eqref{eq:BVA_bias}. Note that, without resorting to the previous change of variable, the bias can be written as
\begin{equation}
    B_a(\F)=R\left(g_{\theta_a^\star}\right)-R\left(f_{\theta^\star}\right)+\left\Vert\theta_a^\star\right\Vert _{ \Gamma}^{2}.
\end{equation}
By linearity of  $g_{\theta_a^\star}$, $R\left(g_{\theta_a^\star}\right)\geq R\left(f_{\theta^\star}\right)=R^\star(\F)$ (because $g_{\theta_a^\star}\in\F$). 

Thus, $\left\Vert\theta_a^\star\right\Vert _{ \Gamma}^{2}\leq B_a(\F)$. Assuming $\mu I\preceq \Gamma$ gives \eqref{eq:BVA_norm}, as
\begin{align*}
\mu \left\Vert\theta_a^\star\right\Vert^2 
    \leq \left\Vert\theta_a^\star\right\Vert _{ \Gamma}^{2}\leq B_a(\F).
\end{align*}
\end{proof}

\subsection{Proof of \Cref{sec:bias}}\label{sec:proofSec3}

We consider the case of imputed-by-0 data, i.e., 
\begin{equation*}
    X_{\rm{\imp}}=P\odot X.
\end{equation*}
Under the MCAR setting
(\cref{ass:MCAR}), 
\begin{equation*}
    \esp \left[X_{\rm{\imp}}|X\right]= H^{-1}X,
\end{equation*}
with $H=\rm{diag}(\rho_1^{-1},...,\rho_d^{-1})$ (variables always missing are discarded) and $(\rho_j)_{j\in [d]}$ the observation rates associated to each input variable. 
\begin{proof}[Proof of \Cref{prop:regularization_implicite}]
    For $i,j\in[d]$,
\begin{align}
\nonumber
\Gamma_{ij}&=\esp\left[\left(\left(X_{\imp}\right)_{i}-\esp\left[\left(X_{\imp}\right)_{i}|X\right]\right)\left(\left(X_{\imp}\right)_{j}-\esp\left[\left(X_{\imp}\right)_{j}|X\right]\right)\right] \\
\nonumber
& =\esp\left[X_{i}X_{j}(P_{i}-\esp P_{i})(P_{j}-\esp P_{j})\right]\\
\nonumber
 &
=\esp\left[X_{i}X_{j}\right]{\rm Cov}(P_{i},P_{j}),\\
&=\Sigma_{ij} V_{ij} \label{eq:gamma=sigmaV}
\end{align}
since $P$ and $X$ are independent and with $V$ defined in \Cref{prop:regularization_implicite}. Therefore, applying \cref{lem:Bias_variance_X'}
with $\Gamma=\Sigma\odot V$ proves the first part of \Cref{prop:regularization_implicite}. Regarding the second part,  under the Ho-MCAR assumption, one has $V=\rho(1-\rho)I$, thus $\Gamma=\rho(1-\rho)\mathrm{diag}(\Sigma)$. Furthermore, if $L^2=\ell^2$, then $\mathrm{diag}(\Sigma)=L^2I$ which gives $\Gamma=L^2\rho(1-\rho)I$. 
\end{proof}

\begin{proof}[Proof of \Cref{prop:biasBernoulliHomo,prop:biasMCAR}]
Under \Cref{ass:MCAR}, since $H$ is a diagonal matrix,
\[
H^{\top}\Gamma H=\Sigma\odot C,
\]
where $C$ is defined in \cref{eq:corrMask}. 
\begin{itemize}
    \item Under \Cref{ass:bernoulli_model}, the matrix $C$ satisfies $ C=\frac{1-\rho}{\rho}I$. Moreover, under \Cref{ass:2moment_sup} (resp.\ \Cref{ass:2moment_inf}), one has  $\Sigma\odot C\preceq \frac{1-\rho}{\rho} L^2 I= \lambda_{\imp}$ (resp.\ $\Sigma\odot C\succeq \frac{1-\rho}{\rho} \ell^2 I=\lambda_{\imp}'$) using \eqref{eq:BVA_risk_star}, we obtain 
    \begin{equation*}
        \inf_{\theta\in\R^d} \left\{R\left(\theta\right)+\lambda_{\imp}'\left\Vert\theta\right\Vert _{2}^{2}\right\}\leq R_{\imp}^\star \leq \inf_{\theta\in\R^d} \left\{R\left(\theta\right)+\lambda_{\imp}\left\Vert\theta\right\Vert _{2}^{2}\right\}.
    \end{equation*}
    Subtracting $R^\star(\F)$, one has
    \begin{equation*}
        B_{\mathrm{ridge},\lambda_{\imp}'}\leq B_\imp\leq B_{\mathrm{ridge},\lambda_{\imp}},
    \end{equation*}
which concludes the proof of \Cref{prop:biasBernoulliHomo}.
\item Under \cref{ass:MCAR}, we have $H\Gamma H^{\top}=\Sigma\odot C$. Using \cref{lem:hadamard_norm}, we obtain for all $\theta $, 
\begin{equation*}
    \left\Vert\theta\right\Vert_{H\Gamma H^{\top}}^2=\left\Vert\theta\right\Vert_{\Sigma\odot C}^2\leq \lambda_{\rm{max}}(C)\left\Vert\theta\right\Vert_{\rm{diag}(\Sigma)}^2.
\end{equation*}

Under \Cref{ass:2moment_sup}, we have $\mathrm{diag}(\Sigma)\preceq L^2I$, thus 
\begin{equation*}
    \left\Vert\theta\right\Vert_{H \Gamma H^{\top}}^2\leq  L^2\lambda_{\rm{max}}(C)\left\Vert\theta\right\Vert_{2}^2.
\end{equation*}
This shows that $\lambda_{\mathrm{max}}(H\Gamma H^{\top})\leq L^2\lambda_{\rm{max}}(C)=\Lambda_{\imp} $
We conclude on \cref{prop:biasMCAR} using \cref{eq:BVA_risk_star}. 
\end{itemize}
\end{proof}
\subsection{Proof of \Cref{lem:norm_control}}
\begin{proof}
Using \eqref{eq:gamma=sigmaV}, we have $\Gamma=V\odot \Sigma$. Using that $\lambda_{\mathrm{min}}(V)I \preceq V$, by \Cref{lem:hadamard_monotonicity}, we obtain 
\begin{equation*}
    \lambda_{\mathrm{min}}(V)I\odot \Sigma \preceq \Gamma,
\end{equation*}
and equivalently $\lambda_{\mathrm{min}}(V)\odot \mathrm{diag}(\Sigma) \preceq \Gamma$.
Under \cref{ass:2moment_inf}, we have $\ell^2 I\preceq\mathrm{diag}(\Sigma)$, thus 
\begin{equation*}
    \ell^2\lambda_{\mathrm{min}}(V) I \preceq \Gamma.
\end{equation*}
Therefore, $\lambda_{\rm min}(\Gamma)\geq \ell^2\lambda_{\mathrm{min}}(V)$. Thus, using \eqref{eq:BVA_norm}, we obtain the first part of \cref{lem:norm_control}:
    \begin{equation}
            \ell^2\lambda_{\min}(V)\left\Vert \theta_{\imp}^\star\right\Vert _{2}^{2}\leq B_{\imp}(\F). 
        \end{equation}
Under \cref{ass:bernoulli_model}, $\lambda_{\mathrm{min}}(V)=\rho(1-\rho)$, so that 
\begin{equation}
            \ell^2\rho(1-\rho)\left\Vert \theta_{\imp}^\star\right\Vert _{2}^{2}\leq B_{\imp}(\F),
\end{equation}
which proves the second part of \cref{lem:norm_control}.

\end{proof}

\section{Stochastic gradient descent}
\subsection{Proof of \cref{thm:SGD_bound}}
\begin{lemma}\label{lem:bach_moulines}
   Assume $\left(x_n, \xi_n\right) \in \mathcal{H} \times \mathcal{H}$ are $\mathcal{F}_n$-measurable for a sequence of increasing $\sigma$-fields $\left(\mathcal{F}_n\right)$, $n \geqslant 1$. Assume that $\mathbb{E}\left[\xi_n \mid \mathcal{F}_{n-1}\right]=0, \mathbb{E}\left[\left\|\xi_n\right\|^2 \mid \mathcal{F}_{n-1}\right]$ is finite and $\mathbb{E}\left[\left(\left\|x_n\right\|^2 x_n \otimes x_n\right) \mid \mathcal{F}_{n-1}\right] \preccurlyeq R^2 H$, with $\mathbb{E}\left[x_n \otimes x_n \mid \mathcal{F}_{n-1}\right]=H$ for all $n \geqslant 1$, for some $R>0$ and invertible operator $H$. Consider the recursion $\alpha_n=\left(I-\gamma x_n \otimes x_n\right) \alpha_{n-1}+\gamma \xi_n$, with $\gamma R^2 \leqslant 1$. Then:
$$
\left(1-\gamma R^2\right) \mathbb{E}\left[\left\langle\bar{\alpha}_{n-1}, H \bar{\alpha}_{n-1}\right\rangle\right]+\frac{1}{2 n \gamma} \mathbb{E}\left\|\alpha_n\right\|^2 \leqslant \frac{1}{2 n \gamma}\left\|\alpha_0\right\|^2+\frac{\gamma}{n} \sum_{k=1}^n \mathbb{E}\left\|\xi_k\right\|^2 .
$$  
\end{lemma}

\begin{proof}
    The idea is to use \cref{lem:bach_moulines} with 
    \begin{itemize}
        \item $x_k=X_{\imp,k}$, $y_k=Y_k$
        \item $H=\Sigma_{\imp}=\esp\left[X_{\imp,k}X_{\imp,k}^\top\right]= \Sigma_P\odot\Sigma$ where $\Sigma_P=\esp\left[PP^\top\right]$
        \item $\alpha_k=\theta_{\imp,k}-\theta_{\imp}^\star$
        \item $\xi_k= X_{\imp,k}(Y_k-X_{\imp,k}^\top \theta_{\imp}^\star)$
        \item $\gamma=\frac{1}{2R^2\sqrt{n}}$
        \item $R^2=\kappa\mathrm{Tr}(\Sigma)$
    \end{itemize}
We can show, with these notations, that recursion \eqref{eq:SGDiteration} leads to recursion $\alpha_n=\left(I-\gamma x_n \otimes x_n\right) \alpha_{n-1}+\gamma \xi_n$ with $\alpha_0=\theta_0- \theta_{\imp}^\star$. Now, let's check the assumption of \cref{lem:bach_moulines}. 
\begin{itemize}
    \item Let show that $
\esp\left[X_{\imp}X_{\imp}^{\top}\left\Vert X_{\imp}\right\Vert _{2}^{2}\right]\preceq R^{2}\Sigma_{\imp}
$. Indeed,
\[
\esp\left[X_{\imp}X_{\imp}^{\top}\left\Vert X_{\imp}\right\Vert _{2}^{2}\right]\preceq\esp\left[X_{\imp}X_{\imp}^{\top}\left\Vert X\right\Vert _{2}^{2}\right],
\]

using that $\left\Vert X_{\imp}\right\Vert _{2}^{2}\leq\left\Vert X\right\Vert _{2}^{2}$, and $0\preccurlyeq X_{\imp}X_{\imp}^{\top}$.
Then, 
\begin{align*}
\esp\left[X_{\imp}X_{\imp}^{\top}\left\Vert X\right\Vert _{2}^{2}\right] & =\esp\esp\left[X_{\imp}X_{\imp}^{\top}\left\Vert X\right\Vert _{2}^{2}|P\right]\\
 & =\esp\esp\left[PP^{\top}\odot XX^{\top}\left\Vert X\right\Vert _{2}^{2}|P\right]\\
 & =\esp\left[\Sigma_{P}\odot XX^{\top}\left\Vert X\right\Vert _{2}^{2}\right]\\
 & =\Sigma_{P}\odot\left(\esp\left[XX^{\top}\left\Vert X\right\Vert _{2}^{2}\right]\right).
\end{align*}

According to \cref{ass:SGD}, $\esp\left[XX^{\top}\left\Vert X\right\Vert _{2}^{2}\right]\preceq R^{2}\Sigma$,
and \cref{lem:hadamard_monotonicity} lead to 
\[
\esp\left[X_{\imp}X_{\imp}^{\top}\left\Vert X_{\imp}\right\Vert _{2}^{2}\right]\preceq R^{2}\Sigma_{P}\odot\Sigma=R^{2}\Sigma_{\imp}.
\]
\item Define $\epsilon_{\imp}=Y-X_{\imp}^{\top}\theta_{\imp}^{\star}=X^{\top}\theta^{\star}+\epsilon-X_{\imp}^{\top}\theta_{\imp}^{\star}$
. First, we have $\epsilon_{\imp}^{2}\leq3\left(\left(X^{\top}\theta^{\star}\right)^{2}+\epsilon^{2}+\left(X_{\imp}^{\top}\theta_{\imp}^{\star}\right)^{2}\right)$,
then \\
\begin{align*}
\esp\left[\left\Vert \xi\right\Vert _{2}^{2}\right] & =\esp\left[\epsilon_{\imp}^{2}\left\Vert X_{\imp}\right\Vert _{2}^{2}\right]\\
 & \leq3\esp\left[\left(\left(X^{\top}\theta^{\star}\right)^{2}+\epsilon^{2}+\left(X_{\imp}^{\top}\theta_{\imp}^{\star}\right)^{2}\right)\left\Vert X_{\imp}\right\Vert _{2}^{2}\right]\\
 & \leq3\left(\esp\left[\left(X^{\top}\theta^{\star}\right)^{2}\left\Vert X\right\Vert _{2}^{2}\right]+\esp\left[\epsilon^{2}\left\Vert X\right\Vert _{2}^{2}\right] \right.\\
 & \qquad \left. +\esp\left[\left(X_{\imp}^{\top}\theta_{\imp}^{\star}\right)^{2}\left\Vert X_{\imp}\right\Vert _{2}^{2}\right]\right).
\end{align*}

Let remark that, using \cref{ass:SGD}
\begin{align*}
\esp\left[\left(X^{\top}\theta^{\star}\right)^{2}\left\Vert X\right\Vert _{2}^{2}\right] & =\esp\left[\mathrm{\theta^{\star}}^{\top}\left(XX^{\top}\left\Vert X\right\Vert _{2}^{2}\right)\theta^{\star}\right]\left\Vert \theta^{\star}\right\Vert _{\Sigma}^{2}\\
 & \leq R^{2}\mathrm{\theta^{\star}}^{\top}\Sigma\theta\\
 & = R^{2}\left\Vert \theta^{\star}\right\Vert _{\Sigma}^{2}.
\end{align*}

Using the first point, by the same way, $\esp\left[\left(X_{\imp}^{\top}\theta_{\imp}^{\star}\right)^{2}\left\Vert X_{\imp}\right\Vert _{2}^{2}\right]\leq\left\Vert \theta_{\imp}^{\star}\right\Vert _{\Sigma_{\imp}}^{2}.$
By \cref{ass:SGD}, we have also than $\esp\left[\epsilon^{2}\left\Vert X\right\Vert _{2}^{2}\right]\leq\sigma^{2}R^{2}$.
Thus, 
\begin{align*}
\esp\left[\left\Vert \xi\right\Vert _{2}^{2}\right] & \leq3R^{2}\left(\sigma^{2}+\left\Vert \theta^{\star}\right\Vert _{\Sigma}^{2}+\left\Vert \theta_{\imp}^{\star}\right\Vert _{\Sigma_{\imp}}^{2}\right)\\
 & \leq3R^{2}\left(\sigma^{2}+2\left\Vert \theta^{\star}\right\Vert _{\Sigma}^{2}\right),
\end{align*}
because $\left\Vert \theta^{\star}\right\Vert _{\Sigma}^{2}=R\left(\theta^{\star}\right)\leq R_{\imp}\left(\theta_{\imp}^{\star}\right)=\left\Vert \theta_{\imp}^{\star}\right\Vert _{\Sigma_{\imp}}^{2}$.
\end{itemize}
Consequently we can apply \cref{lem:bach_moulines}, to obtain 
\begin{align*}
    &\left(1-\frac{1}{2\sqrt{n}}\right) \mathbb{E}\left[\left\langle{\bar\theta_{\imp,n}-\theta_{\imp}^\star}, \Sigma_{\imp}(\bar\theta_{\imp,n}-\theta_{\imp}^\star) \right\rangle\right]+\frac{1}{2 n \gamma} \mathbb{E}\left\|\theta_{\imp,n}-\theta_{\imp}^\star\right\|^2 \\
    &\leqslant \frac{1}{2 n \gamma}\left\|\theta_{\imp}^\star-\theta_0\right\|^2+\frac{\gamma}{n} \sum_{k=1}^n \mathbb{E}\left\|\xi_k\right\|^2 .
\end{align*}
The choice $\gamma=\frac{1}{2R^2\sqrt{n}}$ leads to 
\begin{equation*}
     \mathbb{E}\left\|\bar\theta_{\imp,n}-\theta_{\imp}^\star\right\|_{\Sigma_{\imp}}^2 \leqslant \frac{2R^2}{ \sqrt{n}}\left\|\theta_{\imp}^\star -\theta_0\right\|^2+4\frac{\sigma^{2}+2\left\Vert \theta^{\star}\right\Vert _{\Sigma}^{2}}{\sqrt{n}}  .
\end{equation*}
We conclude on \Cref{thm:SGD_bound} using that, 
\begin{align*}
    \esp\left[R_{\imp}\left(\bar{\theta}_{\imp}\right)\right]-R^{\star}
    &= \esp\left[R_{\imp}\left(\bar{\theta}_{\imp}\right)\right]-R_{\imp}^\star +R_{\imp}^\star-R^{\star}\\
    &= \mathbb{E}\left\|\bar\theta_{\imp,n}-\theta_{\imp}^\star\right\|_{\Sigma_{\imp}}^2+ B_{\imp}. 
\end{align*}

\end{proof}

\subsection{Proof of \Cref{prop:SGD_bound_bernoulli} and \Cref{cor:SGD_upp_bound}}
\begin{proof}[Proof of \Cref{prop:SGD_bound_bernoulli}]
    First, under \cref{ass:2moment_sup}, $\mathrm{Tr}(\Sigma)\leq dL^2$. Then, initial conditions term with $\theta_0=0$, 
\begin{equation}
    \frac{\kappa\mathrm{Tr}(\Sigma)}{\sqrt{n}}\left\Vert \theta^{\star}_{\imp}\right\Vert _{2}^{2}\leq  \frac{\kappa L^2d}{\sqrt{n}\ell^2\rho(1-\rho) }B_{\imp}(\F),
\end{equation}
using \cref{lem:norm_control}. We obtain \Cref{prop:SGD_bound_bernoulli} using inequality above in \Cref{thm:SGD_bound}. 
\end{proof}
\begin{proof}[proof of \Cref{cor:SGD_upp_bound}]
    We obtain the upper bounds considered that: according to \Cref{prop:biasBernoulliHomo}, $B_{\imp}\leq B_{\mathrm{ridge},\lambda_{\imp}}$; under \Cref{ass:2moment_inf}, $\mathrm{Tr}(\Sigma)\geq d\ell^2$. Then, we put together \Cref{prop:SGD_bound_bernoulli} and ridge bias bound (see \Cref{sec:proof_examples}). 
\end{proof}

\subsection{Miscellaneous}
\begin{propo}
    If $X$ statisfies $\esp\left[XX^{\top}\left\Vert X\right\Vert _{2}^{2}\right]\preceq \kappa \mathrm{Tr}(\Sigma)\Sigma$, then $\esp\left[\epsilon^{2}\left\Vert X\right\Vert _{2}^{2}\right]\leq\sigma^{2}\kappa\mathrm{Tr}(\Sigma)$ with $\sigma^2\leq 2\esp[Y^2]+2\esp[Y^4]^{1/2}$.
\end{propo}
\begin{proof}
    \begin{align*}
\esp\left[\epsilon^{2}\left\Vert X\right\Vert _{2}^{2}\right] & =\esp\left[\left(Y-X^{\top}\theta^{\star}\right)^{2}\left\Vert X\right\Vert _{2}^{2}\right]\\
 & \leq2\esp\left[\left(\left(X^{\top}\theta^{\star}\right)^{2}+Y^{2}\right)\left\Vert X\right\Vert _{2}^{2}\right]\\
 & \leq2\esp\left[Y^{2}\left\Vert X\right\Vert _{2}^{2}\right]+2\esp\left[\left(X^{\top}\theta^{\star}\right)^{2}\left\Vert X\right\Vert _{2}^{2}\right].
\end{align*}
Regarding the first term, by Cauchy Schwarz, 
\begin{align*}
\esp\left[Y^{2}\left\Vert X\right\Vert _{2}^{2}\right]^{2} & \leq\esp\left[Y^{4}\right]\esp\left[\left\Vert X\right\Vert _{2}^{4}\right]\\
 & \leq\esp\left[Y^{4}\right]\esp\left[\mathrm{Tr}\left(XX^{\top}\left\Vert X\right\Vert _{2}^{2}\right)\right]\\
 & \leq\esp\left[Y^{4}\right]\kappa\mathrm{Tr}(\Sigma)^{2}.
\end{align*}
As for the second term,
\begin{align*}
\esp\left[\left(X^{\top}\theta^{\star}\right)^{2}\left\Vert X\right\Vert _{2}^{2}\right] & =\esp\left[(\theta^{\star})^{\top}XX^{\top}\left\Vert X\right\Vert _{2}^{2}\theta^{\star}\right]\\
 & \leq\kappa\mathrm{Tr}(\Sigma)\esp\left[(\theta^{\star})^{\top}\Sigma\theta^{\star}\right]\\
 &\leq\kappa\mathrm{Tr}(\Sigma)\left\Vert \theta^{\star}\right\Vert _{2}^{2}.
\end{align*}
\[
\esp\left[\epsilon^{2}\left\Vert X\right\Vert _{2}^{2}\right]\leq\esp\left[Y^{4}\right]^{\frac{1}{2}}\kappa\mathrm{Tr}(\Sigma)+\kappa\mathrm{Tr}(\Sigma)\left\Vert \theta^{\star}\right\Vert _{\Sigma}^{2}\leq\sigma^{2}\kappa\mathrm{Tr}(\Sigma)\left\Vert \theta^{\star}\right\Vert _{\Sigma}^{2}.
\]
\end{proof}

\section{Details on examples}\label{sec:proof_examples}
Recall that 
\begin{align}
B_{{\rm {ridge},\lambda}}(\F)  &=\lambda \left\Vert\theta^{\star}\right\Vert _{\Sigma(\Sigma+\lambda I)^{-1}}^{2}\\&= \lambda\sum_{j=1}^{d}\frac{\lambda_{j}}{\lambda_{j}+\lambda}(v_{j}^{\top}\theta^{\star})^{2}. 
\end{align}

\subsection{Low-rank covariance matrix (\cref{ex:low_rank}) }

\begin{propo}[Low-rank covariance matrix with equal singular values]
Consider a covariance matrix with a low rank $r\ll d$ and constant eigenvalues ($\lambda_1=\lambda_2=...=\lambda_r$). Then, 
\begin{equation*}
    B_{\mathrm{ridge},\lambda}(\F)=\lambda\frac{r}{\mathrm{Tr}(\Sigma)}\left\Vert \theta^{\star}\right\Vert _{\Sigma}^{2}.
\end{equation*}
\end{propo}
\begin{proof}
    Using that $\lambda_1=\dots=\lambda_r$ and $\sum_{j=1}^r \lambda_j=\mathrm{Tr}(\Sigma)$, we have $\lambda_1=\dots=\lambda_r=\frac{\mathrm{Tr}(\Sigma)}{r}$. Then $\Sigma(\Sigma+\lambda I)^{-1}\preceq \lambda_r^{-1} \Sigma=\frac{r}{\mathrm{Tr}(\Sigma)}\Sigma$. Thus, 
    \begin{equation*}
    B_{\mathrm{ridge},\lambda}(\F)=\lambda \left\Vert\theta^{\star}\right\Vert _{\Sigma(\Sigma+\lambda I)^{-1}}^{2}= \lambda\frac{r}{\mathrm{Tr}(\Sigma)}\left\Vert \theta^{\star}\right\Vert _{\Sigma}^{2}.
\end{equation*}
\end{proof}

\subsection{Low-rank covariance matrix compatible with $\theta^\star$ (\cref{ex:low_rank2})}

\begin{propo}[Low-rank covariance matrix compatible with $\theta^\star$]
Consider a covariance matrix with a low rank $r\ll d$ and assume that $\langle \theta^\star,v_1\rangle^2\geq \dots\geq \langle \theta^\star,v_d\rangle^2 $, then 
\begin{equation*}
    B_{\mathrm{ridge},\lambda}(\F)\lesssim \lambda\frac{r(\log(r)+1)}{\mathrm{Tr}(\Sigma)}\left\Vert \theta^{\star}\right\Vert _{\Sigma}^{2}.
\end{equation*}
 
\end{propo}
\begin{proof}
    Recall that
\begin{equation}\label{eq:espY^2}
    \Vert \theta^\star\Vert_\Sigma^2= \sum_{j=1}^d \lambda_j(v_j^\top\theta^{\star})^2.
\end{equation}
Under the assumptions of \cref{ex:low_rank2}, using that $(\lambda_j)_j$ and $\left((v_j^\top\theta^{\star})^2\right)_j$ are decreasing, then for all $k\in [r]$, 
\begin{equation*}
    \sum_{j=1}^k \lambda_j(v_k^\top\theta^{\star})^2\leq  \Vert \theta^\star\Vert_\Sigma^2.
\end{equation*}
Thus, for all $k\in [r]$,
\begin{equation*}
    (v_k^\top\theta^{\star})^2\leq \frac{\Vert \theta^\star\Vert_\Sigma^2}{\sum_{j=1}^k \lambda_j} .
\end{equation*}
Using that $\sum_{j=1}^r \lambda_j=\mathrm{Tr}(\Sigma)$ and that eigenvalues are decreasing, we have $\sum_{j=1}^k \lambda_j\geq \frac{k}{r}\mathrm{Tr}(\Sigma)$ using \Cref{lem:somme_suite}. Then 
\begin{align*}
B_{{\rm {ridge},\lambda}}(\F) & =\lambda\sum_{k=1}^{r}\frac{\lambda_{k}}{\lambda_{k}+\lambda}(v_{k}^{\top}\theta^{\star})^{2}\\
 & \leq\lambda\sum_{k=1}^{r}(v_{k}^{\top}\theta^{\star})^{2}\\
 & \leq \lambda \Vert \theta^\star\Vert_\Sigma^2 \sum_{k=1}^r \frac{1}{\sum_{j=1}^k\lambda_j} \\
 & \leq\lambda\sum_{k=1}^{r}\frac{r}{k\mathrm{Tr}(\Sigma)}\\
 & \leq\lambda\frac{r}{\mathrm{Tr}(\Sigma)}\sum_{k=1}^{r}\frac{1}{k}\\
 & \lesssim\lambda\frac{r}{\mathrm{Tr}(\Sigma)}(\log(r)+1),
\end{align*}
by upper-bounding the Euler-Maclaurin formula. 
\end{proof}

\subsection{Spiked covariance matrix (\Cref{ex:noised_low_rank})}
\begin{propo}[Spiked model]
Assume that the covariance matrix is decomposed as $\Sigma= \Sigma_{\leq r}+ \Sigma_{> r} $. Suppose that $\Sigma_{> r}\preceq \eta I$ (small operator norm) and that all non-zero eigenvalues of $\Sigma_{\leq r}$ are equal, then   
\[
B_{{\rm {ridge},\lambda}}(\F)\leq\frac{r}{\mathrm{Tr}(\Sigma)-d\eta}\left\Vert \theta^{\star}\right\Vert _{\Sigma}^{2}+\eta\left\Vert \theta_{>}^{\star}\right\Vert _{2}^{2}.
\]
where $\theta^{\star}_{>r}$ is the projection of $\theta^{\star}$ on the range of $\Sigma_{> r}$. 
\end{propo}
\begin{proof} One has
\begin{align*}
\Sigma(\Sigma+\lambda I)^{-1} & =\Sigma_{\leq}(\Sigma+\lambda I)^{-1}+\Sigma_{>}(\Sigma+\lambda I)^{-1}\\
 & \preceq\Sigma_{\leq}(\Sigma_{\leq}+\lambda I)^{-1}+\Sigma_{>}(\Sigma_{>}+\lambda I)^{-1}\\
 & \preceq\frac{1}{\mu}\Sigma_{\leq}+\frac{1}{\lambda}\Sigma_{>}
\end{align*}
where $\mu$ is the non-zero eigenvalue of $\Sigma_{\leq}$. Thus,
\begin{align*}
B_{{\rm {ridge},\lambda}}(\F) & =\left\Vert \theta^{\star}\right\Vert _{\lambda\Sigma(\Sigma+\lambda I)^{-1}}^{2}\\
 & \leq\left\Vert \theta^{\star}\right\Vert _{\frac{\lambda}{\mu}\Sigma_{\leq}+\Sigma_{>}}^{2}\\
 & \leq\frac{\lambda}{\mu}\left\Vert \theta^{\star}\right\Vert _{\Sigma}^{2}+\left\Vert \theta^{\star}\right\Vert _{\Sigma_{>}}^{2}.
\end{align*}

Using that $\lambda_{\mathrm{max}}(\Sigma_{>})\leq\eta$, we have
$$B_{{\rm {ridge},\lambda}}(\F)\leq\frac{\lambda}{\mu}\left\Vert \theta^{\star}\right\Vert _{\Sigma}^{2}+\eta\left\Vert \theta_{>}^{\star}\right\Vert _{2}^{2}.$$

Using Weyl's inequality, for all $j\in[d]$, $\lambda_{j}(\Sigma_{\leq}+\Sigma_{>})\leq\lambda_{j}(\Sigma_{\leq})+\eta$.
Summing the previous inequalities, we get 
\[
\mathrm{Tr(\Sigma)}\leq r\mu+d\eta.
\]
Thus, 
\[
\mu\geq\frac{\mathrm{Tr}(\Sigma)-d\eta}{r}.
\]
In consequence, 
\[
B_{{\rm {ridge},\lambda}}(\F)\leq\frac{r}{\mathrm{Tr}(\Sigma)-d\eta}\left\Vert \theta^{\star}\right\Vert _{\Sigma}^{2}+\eta\left\Vert \theta_{>}^{\star}\right\Vert _{2}^{2}.
\]
    
\end{proof}

\section{Technical lemmas }

\begin{lemma}\label{lem:hadamard_monotonicity}
   Let $A,B,V$ be three symmetric non-negative matrix, if $A\preceq B$
then $A\odot V\preceq B\odot V$. 
\end{lemma}
 \begin{proof}
     Let $X\sim\mathcal{N}(0,V)$ and $\theta\in\R^d$,

\begin{align*}
\left\Vert \theta\right\Vert _{A\odot V}^{2} & =\theta^{\top}A\odot V\theta\\
 & =\theta^{\top}\left(\left(\esp XX^{\top}\right)\odot A\right)\theta\\
 & =\esp\left[\theta^{\top}\left(\left(XX^{\top}\right)\odot A\right)\theta\right]\\
 & =\esp\left[\sum_{i,j}\theta_{i}\left(\left(XX^{\top}\right)\odot A\right)_{ij}\theta_{j}\right]\\
 & =\esp\left[\sum_{i,j}\theta_{i}X_{i}X_{j}A_{ij}\theta_{j}\right]\\
 & =\esp\left[\sum_{i,j}\left(\theta_{i}X_{i}\right)\left(\theta_{j}X_{j}\right)A_{ij}\right]\\
 & =\esp\left[\left\Vert X\odot\theta\right\Vert _{A}^{2}\right]\\
 & \leq\esp\left[\left\Vert X\odot\theta\right\Vert _{B}^{2}\right]\\
 & =\left\Vert \theta\right\Vert _{B\odot V}^{2}
\end{align*}
 \end{proof}

\begin{lemma}\label{lem:hadamard_norm}

Let $A,B$ be two non-negative symmetric matrices, then $A\odot B$ is non-negative symmetric and, for all $\theta\in\R^d$:
\begin{equation*}
    \left\Vert\theta\right\Vert_{A\odot B}^2\leq \lambda_{\rm{max}}(B)\left\Vert\theta\right\Vert_{\rm{diag}(A)}^2,
\end{equation*}
where $\rm{diag}(A)$ is the diagonal matrix containing the diagonal terms of $A$.
\end{lemma}
\begin{proof}
Let $X\sim \mathcal{N}(0,A)$, thus $A=\esp\left[XX^\top\right]$, then for $\theta\in\mathbb{R}^d$

\begin{align*}
\left\Vert \theta\right\Vert _{A\odot B}^{2} & =\theta^{\top}A\odot B\theta\\
 & =\theta^{\top}\left(\left(\esp XX^{\top}\right)\odot B\right)\theta\\
 & =\esp\left[\theta^{\top}\left(\left(XX^{\top}\right)\odot B\right)\theta\right]\\
 & =\esp\left[\sum_{i,j}\theta_{i}\left(\left(XX^{\top}\right)\odot B\right)_{ij}\theta_{j}\right]\\
 & =\esp\left[\sum_{i,j}\theta_{i}X_{i}X_{j}B_{ij}\theta_{j}\right]\\
 & =\esp\left[\sum_{i,j}\left(\theta_{i}X_{i}\right)\left(\theta_{j}X_{j}\right)B_{ij}\right]\\
 & =\esp\left[\left(X\odot\theta\right)^{\top}B\left(X\odot\theta\right)\right]\\
 & \geq 0, 
\end{align*}
using that $B$ is positive. Thus $A\odot B$ is positive. Furthermore,
\begin{align*}
\left\Vert \theta\right\Vert _{A\odot B}^{2} &=\esp\left[\left(X\odot\theta\right)^{\top}B\left(X\odot\theta\right)\right]\\
 & \leq\lambda_{{\rm max}}(B)\esp\left[\left(X\odot\theta\right)^{\top}\left(X\odot\theta\right)\right]\\
 & =\lambda_{{\rm max}}(B)\esp\left[\sum_{i}\theta_{i}^{2}X_{i}^{2}\right]\\
 & =\lambda_{{\rm max}}(B)\sum_{i}\theta_{i}^2\esp\left[X_{i}^{2}\right]\\
 & =\lambda_{{\rm max}}(B)\sum_{i}\theta_{i}^{2}A_{ii}\\
 & =\lambda_{{\rm max}}(B)\left\Vert \theta\right\Vert _{{\rm diag}(A)}^{2}.
\end{align*}
\end{proof}

\begin{lemma}\label{lem:somme_suite}
 Let $(v_{j})_{j\in[d]}$a non-decreasing sequence of positive number,
and $S=\sum_{j=1}^{d}v_{j}$, for all $k\in[d]$,
\[
\sum_{j=1}^{k}v_{j}\geq\frac{k}{d}S.
\] 
\end{lemma}
\begin{proof} 
We use a absurd m, if $\sum_{j=1}^{k}v_{j}<\frac{k}{d}S$. Then, using
that $(v_{j})_{j\in[d]}$are non-decreasing,
\[
kv_{k}<\frac{k}{d}S.
\]
Thus $v_{k+1}<\frac{1}{d}S$, summing last elements, 
\[
\sum_{j=r+1}^{d}v_{j}<\frac{d-r}{d}S.
\]
Then, 
\[
S=\sum_{j=1}^{k}v_{j}=\sum_{j=1}^{r}v_{j}+\sum_{j=r+1}^{d}v_{j}<\frac{k}{d}S+\frac{d-r}{d}S=S.
\]
Thus, this is absurd. 
\end{proof}

\end{document}